\documentclass{amsart}
\usepackage[english]{babel}
\usepackage{enumitem}
\usepackage{amsmath}
\usepackage{graphicx}
\usepackage{amsaddr}
\usepackage{caption}
\usepackage{tikz}
\usepackage{subcaption}
\usepackage{algorithm}
\usepackage{algpseudocode}
\usepackage[colorlinks=true, allcolors=blue]{hyperref}
\usepackage{amsthm,amsfonts,bbm,bm, color}

\newtheorem{thm}{Theorem}

\newtheorem{lem}[thm]{Lemma}
\newtheorem{rmk}[thm]{Remark}
\newtheorem{prop}[thm]{Proposition}
\usepackage{cite}
\allowdisplaybreaks

\newcommand{\bs}{\ensuremath{\boldsymbol}}

\newcommand{\R}{\ensuremath{\mathbb{R}}}

\title[DeepWKB: Learning WKB of Stochastic Systems]{DeepWKB: Learning WKB Expansions of Invariant Distributions for Stochastic Systems}
\author{Yao Li}
\address{Department of Mathematics and Satistics, University of Massachusetts Amherst}
\email{liyao@umass.edu}

\author{Yicheng Liu}
\address{Department of Mathematics, Mount Holyoke College}
\email{liu235y@mtholyoke.edu}

\author{Shirou Wang}
\address{School of Mathematics, Jilin University\\
$\&$\\
Department of Mathematical and Statistical Sciences, University of Alberta}
\email{shirou@jlu.edu.cn}

 

\thanks{Y. Li was partially supported by NSF DMS-2108628. 
S. Wang was partially supported by NSFC 12201244, NSFC 12271204, and a faculty development grant from Jilin University.}

\begin{document}

\begin{abstract}
This paper introduces a novel deep learning method, called {\it DeepWKB}, for estimating  the invariant distribution  of randomly perturbed systems via its Wentzel–Kramers–Brillouin (WKB) approximation
\[u_\epsilon(\bs x)=Q(\epsilon)^{-1}Z_\epsilon(\bs x)\exp\{-V(\bs x)/\epsilon\},\]
where $V$ is known as the quasi-potential, $\epsilon$ denotes the noise strength, and $Q(\epsilon)$ is the normalization factor.
By utilizing both Monte Carlo  data and the partial differential equations satisfied by  $V$ and  $Z_\epsilon$, 
the DeepWKB method computes $V$ and  $Z_\epsilon$ separately. This enables an approximation of the invariant distribution in the singular regime where $\epsilon$ is sufficiently small, which remains a significant challenge for most existing methods. 
Moreover, the DeepWKB method is applicable to higher-dimensional stochastic systems whose deterministic counterparts admit non-trivial attractors. In particular, it
provides a scalable and flexible alternative for computing the quasi-potential, which plays a key role in the analysis of rare events, metastability, and the stochastic stability of complex systems.
\end{abstract}
\maketitle

\section{Introduction}
Random perturbations of dynamical systems, often modeled by stochastic differential equations, arise in a wide range of scientific and engineering applications. Beyond their practical importance, such  systems have been extensively studied from a theoretical perspective, particularly in connections with rare events, metastability, and long-time dynamical behaviors. Consider a stochastic differential equation (SDE) on $\mathbb R^n$ of the form 
\begin{equation}
  \label{SDE}
  \mathrm{d} X_{t} = f(X_{t}) \mathrm{d}t + \sqrt{\epsilon}\sigma (X_{t})
  \mathrm{d}B_{t},
\end{equation}
where 
$f=(f^i):\mathbb R^n\to \mathbb R^n
$ is a
vector field,
$\sigma:\mathbb R^n\to \mathbb R^{n\times m}$ is a  matrix-valued function, $B_{t}$ is a
standard $m$-dimensional Wiener process, and $\epsilon > 0$ is a small parameter. 
The stochastic system \eqref{SDE} can be regarded as a small-noise perturbation of the deterministic dynamical system governed by the  ordinary differential equation 
\begin{equation}
  \label{ODE}
\mathrm{d}X_{t} = f(X_{t})\mathrm{d}t.
\end{equation}

A natural way to analyze the long-time behavior of the stochastic system \eqref{SDE} is to examine  its invariant distribution, which often captures the limiting dynamics. In this regard, the Fokker-Planck equation  
plays a central role 
as it characterizes
the evolution of  distributions of the  stochastic system. 
Specifically, let $X_t$ be the stochastic process given by \eqref{SDE}. The associated Fokker-Planck equation is a partial differential equation (PDE) 
given by
\begin{eqnarray*}
\partial_tu=-\sum_{i=1}^n\partial_i(f^iu)+\frac{1}{2}\epsilon\sum_{i,j=1}^n\partial_{ij}(a^{ij}u):=\mathcal Lu, 
\end{eqnarray*}
where $u=u(\bs x,t)$ denotes the density function of distribution of $X_t$, and $\partial_i=\partial_{x_i}, \partial_{ij}=\partial_{x_ix_j}$. The matrix 
$A= (a^{ij})= \sigma \sigma ^{T}$ is  called the diffusion matrix.
If $X_t$ admits an invariant distribution $\pi_\epsilon$ with density $u_\epsilon=u_\epsilon(\bs x)$, then $u_\epsilon$ satisfies the stationary Fokker-Planck equation along with the normalization condition 
\begin{eqnarray}
\label{stationary_FPE}
\left\{
\begin{aligned}
& \mathcal Lu_\epsilon  =-\sum_{i = 1}^{n}\partial_i(f^i u_{\epsilon}) + \frac{1}{2}\epsilon
  \sum_{i,j=1}^n\partial_{ij}(a^{ij} u_\epsilon)=0,\\
& \displaystyle\int_{\mathbb R^n} u_\epsilon \mathrm{d}\bs x  =1,\quad u_\epsilon\ge0.
\end{aligned}
\right.
\end{eqnarray}
Throughout this paper, it is assumed that such an invariant distribution   $\pi_\epsilon$ exists uniquely. 

The Fokker-Planck equation rarely admits an analytic solution, especially when the underlying deterministic system \eqref{ODE} is complicated, with occasional exceptions such as the Ornstein–Uhlenbeck process. On the other hand, 
various numerical methods  have been developed to solve the stationary Fokker-Planck equation \eqref{stationary_FPE}, which are generally classified into two categories: the PDE-based approaches and  stochastic sampling methods such as Monte Carlo simulations.
The  PDE-based approaches usually yield solutions with high accuracy but are generally effective only in low dimensions (typically $n\le2$), whereas the Monte Carlo simulations provide greater flexibility by sampling over long trajectories of \eqref{SDE}, but often suffer from low accuracy, particularly in high-dimensional settings. 
In \cite{li2019data}, a hybrid data-driven approach is proposed for computing the invariant density $u_\epsilon$,  combining 
the high accuracy of numerical PDE methods with the flexibility of Monte Carlo simulations. 
By treating the Monte Carlo  solution as ``reference data", the problem is framed within an optimization framework, which was later addressed using  artificial neural networks \cite{zhai2022deep}. However, when $\epsilon \ll 1$, the Fokker-Planck equation becomes highly singular, and most existing numerical methods
fail because the solution strongly concentrates around a low-dimensional submanifold within the domain. 

In applications where $\epsilon \ll 1$, a classical approach for solving singular perturbation problems is  the WKB method, a century-old  technique  named after Wentzel, Kramers, and Brillouin \cite{wentzel1926verallgemeinerung, kramers1926wellenmechanik, brillouin1926mecanique}. In the context of SDEs, the WKB method involves expressing the invariant density function in the following WKB form 
\begin{eqnarray}\label{WKB}
u_\epsilon(\bs x)=Q(\epsilon)^{-1}Z_\epsilon(\bs x)\exp\{-V(\bs x)/\epsilon\}
\end{eqnarray}
where $Q(\epsilon)>0$ is the  normalization factor, and the prefactor $Z_\epsilon$ satisfies
\begin{eqnarray}\label{Z0}
0<Z_0(\bs x):=\lim\nolimits_{\epsilon\to0}Z_\epsilon(\bs x)<\infty.
\end{eqnarray}
 The function $V$ is often called the quasi-potential since it coincides with the Freidlin-Wentzel quasi-potential function in large deviation theory  when it is twice differentiable \cite{freidlin2012random, ludwig1975persistence}.
Although the WKB expression \eqref{WKB}  is rigorously justified only when the 
underlying deterministic 
system admits a simple attractor, such as
a stable equilibrium \cite{freidlin2012random}  or, in certain cases, a limit cycle \cite{day1994regularity,mou2025quasi}, it is often applied to broader settings, including  
non-equilibrium systems with more complex dynamics. In particular, it provides a qualitative insight into the concentration  of  invariant  distributions in the zero-noise limit, which is closely related to the
stochastic stability of  deterministic systems  \cite{huang2015integral, huang2015steady, ji2019quantitative, huang2018concentration,freidlin2012random}.

This paper introduces a novel machine learning framework, termed  {\bf DeepWKB}, to compute the invariant distribution of stochastic  system \eqref{SDE} via its WKB approximations. The approach is similar in spirit to that of \cite{zhai2022deep}, which utilizes Monte Carlo simulation data to estimate the solution of a PDE at a set of collocation points, and then trains an artificial neural network to fit both the Monte Carlo data and the corresponding PDE. The latter part of the method is commonly known as the physics-informed neural network (PINN) \cite{raissi2019physics}. 
Specifically, instead of attempting to solve a singular Fokker-Planck equation with a single neural network, the DeepWKB method trains two separate neural networks for solving two non-singular PDEs which govern the quasi-potential $V$ and prefactor $Z_0$, respectively, thereby yielding a more structured WKB-type solution of \eqref{stationary_FPE}. 
As shown in \cite{zhai2022deep}, the accuracy of Monte Carlo simulations need not  be  high, as the deep neural network solver is highly tolerant of spatially uncorrelated errors in the Monte Carlo data.
Moreover, the DeepWKB approach is applicable regardless of whether the attractor  is a fixed point, a limit cycle, an invariant manifold, or even a strange attractor. In particular, the use of artificial neural networks makes it well-suited to higher-dimensional problems that are usually infeasible for grid-based methods.

A key step in the DeepWKB framework is obtaining the training data for  $V$ and $Z_0$ from the Monte Carlo simulations of invariant distributions. 
While a single Monte Carlo simulation is insufficient to estimate $V$ and $Z_0$,  linear regression on simulations at multiple noises  enables an estimation of $V$ and $Z_0$ in the neighborhood of the attractor. 
The method of characteristics can then extend the estimation to a larger region. In addition, to handle general attractors beyond a simple stable equilibrium, an extra training set is incorporated to enforce $V = 0$ at  the attractor. The DeepWKB method is validated on several examples, including the Van der Pol system with a stable limit cycle, the figure-eight attractor whose invariant  measure $\pi_\epsilon$ converges weakly at a logarithmic rate, 
a chaotic system which admits a strange attractor, and a coupled Van der Pol system admitting a stable invariant manifold. 

We note that  computing the quasi-potential  $V$ is of independent interest,
as it characterizes the difficulty of transitions between states in the small noise regime, which is essential for understanding the first exit problem and metastability of stochastic systems \cite{wang2008potential, wang2015landscape, zhou2012quasi, zhou2016construction,dykman1994large}.
Nevertheless, the computation of $V$ has been a long-standing challenge. Over the past decades, various numerical methods, including the ordered upwind Hamilton-Jacobi solver \cite{dahiya2018ordered}, the minimal action method \cite{weinan2004minimum, heymann2008geometric}, the action plot method \cite{beri2005solution}, and more recently, artificial neural network solvers \cite{li2022machine, lin2022data}, 
have been developed to numerically compute the quasi-potential. 
However, all these methods  either suffer from high computational cost or encounter significant difficulties when the attractor exhibits non-trivial geometric structure. Since the computation of  quasi-potential is a key component of the DeepWKB method, this paper provides a new approach for computing the quasi-potential of stochastic system \eqref{SDE} with higher dimensions, applicable to more general settings where the attractor is not limited to a stable equilibrium or a stable limit cycle.

Finally, it should be noted that the WKB approximation is  theoretically justified  in very limited situations. It is thus necessary to verify whether the invariant  density  $u_\epsilon$ indeed satisfies the WKB form \eqref{WKB} before applying any numerical methods based on it. 
To this end, a statistical method is proposed to assess the validity of the WKB approximation  for invariant distributions.  
Specifically, the probability density function $u_\epsilon$ is regarded as satisfying WKB form statistically if the rescaled residual sum of squares of the simulated data follows a $\chi^{2}$-distribution; see Section \ref{sec:WKB_statistical_validation} for details.

This paper is organized as follows.  Section \ref{sec:pre}  presents preliminaries for the subsequent sections. Section \ref{sec:3}  
presents the linear regression method for estimating the quasi-potential $V$ and prefactor $Z_\epsilon$ and provides a statistical validation tool for the WKB approximation, 
establishing the theoretical foundation of the DeepWKB framework. 
Section \ref{sec:V} and \ref{sec:Z} introduce the artificial neural network training methods for the quasi-potential function $V$ and the prefactor function $Z$, respectively. Section \ref{sec:example} presents several numerical examples, including  statistical validations of their WKB approximations.

\section{Preliminary}\label{sec:pre}
This section reviews key elements for the establishment of the DeepWKB method. These includes the Hamilton–Jacobi equation satisfied by the quasi-potential function and the symplectic numerical scheme  used for expanding the training set.

\subsection{Quasi-potential and Hamilton-Jacobi equation}
In the WKB expansion of invariant densities, both the quasi-potential function $V$ and the prefactor function $Z_\epsilon$ satisfy certain non-singular PDEs. To see this, substituting \eqref{WKB} into the stationary 
Fokker-Planck equation \eqref{stationary_FPE} and using the symmetry of the diffusion $A$ yields
\begin{align*}
0 = &\ \frac{1}{\epsilon}\left ( \sum_{i=1}^{n}f^{i} \partial_i V + \frac{1}{2} \sum_{i,j = 1}^{n} a^{ij} \partial_i V
    \partial_j V\right
    )Z_\epsilon\exp \left (- \frac{1}{\epsilon} V \right)\\
   &-\ \sum_{i = 1}^{n} \left( f^i\partial_i Z_\epsilon + \sum_{j = 1}^{n} a^{ij}
     \partial_{j} V\partial_{i} Z_\epsilon
     +\sum_{i = 1}^{n} \partial_if^i + \frac{1}{2}
     \sum_{i,j = 1}^{n} a^{ij} \partial^{2}_{ij} V + \sum_{i,j =
     1}^{n}\partial_{i} a^{ij}\partial_{j}V
     \right)\cdot\exp\left (- \frac{1}{\epsilon}V \right) \\
 & +  \epsilon \left ( \frac{1}{2}\sum_{i,j = 1}^{n}Z_\epsilon      \partial^{2}_{ij}a^{ij} + \sum_{i,j = 1}^{n}
      \partial_{i} a^{ij} \partial_{j}Z_\epsilon+ \frac{1}{2}\sum_{i,j
      = 1}^{n} a^{ij} \partial^{2}_{ij}Z_\epsilon\right ) \exp \left (-
      \frac{1}{\epsilon} V \right) \\
  := &\left( \frac{1}{\epsilon}
      \mathcal{H}V\cdot Z_\epsilon
      - \mathcal{L}^{0}Z_\epsilon + \epsilon \mathcal{L}^{1}Z_\epsilon
       \right) \exp \left (- \frac{1}{\epsilon} V \right).
\end{align*}

Expanding $Z_\epsilon(\bs x)=\sum_{k=0}^\infty Z_k(\bs x)\epsilon^k$ in powers of $\epsilon$ and matching the coefficients of equal orders,  the following equations for functions $V$ and $Z_k$'s are obtained: 
\begin{equation}
  \label{HJE}
  \mathcal{H}V = \sum_{i=1}^{n}f^{i} \partial_{i} V +
    \frac{1}{2} \sum_{i,j = 1}^{n} a^{ij} \partial_i V
    \partial_j V = 0 \,,
\end{equation}
\begin{eqnarray}\label{transZ}
\mathcal{L}^{0}Z_{0}&=&\sum_{i = 1}^{n} \left( f^{i} + \sum_{j = 1}^{n} a^{ij}
     \partial_j V
     \right ) \partial_i Z_{0}\\
     &\ &+ \left ( \sum_{i = 1}^{n} \partial_if^{i} + \frac{1}{2}
     \sum_{i,j = 1}^{n} a^{ij} \partial^{2}_{ij} V + \sum_{i,j =
     1}^{n}\partial_{i} a^{ij} \partial_{j}V\right )Z_{0}=0\nonumber
\end{eqnarray}
and
\begin{equation*}
\mathcal{L}^{0}Z_{k+1} + \mathcal{L}^{1}Z_{k} = 0,\quad k=0,1,\cdots.
\end{equation*}
The equation \eqref{HJE} is known as the  Hamilton-Jacobi equation, where  $\mathcal{H}$ is referred to as the Hamilton-Jacobi operator. 

The above calculation requires the quasi-potential $V$ to be at least twice differentiable, 
which is generally not guaranteed. 
In fact, results on the regularity of $V$  remain rather limited. It has been known that if the the attractor of the 
deterministic system \eqref{ODE} is a stable equilibrium
or a limit cycle,  $V$ is at least $C^{2}$ in a sufficiently small neighborhood of the attractor \cite{day1994regularity, day1985some}.
Recently, $C^2$ regularity of quasi-potential is established  when the attractor is an invariant manifold \cite{mou2025quasi}. 

\subsection{Symplectic scheme}\label{sec:symplectic scheme} The symplectic scheme will be used in the training set expansion. Consider the Hamiltonian system
\begin{eqnarray}\label{ham}
\left\{
\begin{aligned}
    &\dot{\bs p} = - \nabla_{\bs q} H\\
    &\dot{\bs q} = \nabla_{\bs p} H
\end{aligned}
\right.
\end{eqnarray}
where $\bs p=(p_i), \bs q=(q_i) \in \mathbb{R}^n$.
Given $\bs \xi = (\bs\xi_{p_i}, \bs\xi_{q_i})^T,\bs\eta = (\bs\eta_{p_i}, \bs\eta_{q_i})^T\in\mathbb{R}^{2n}$, the symplectic form is defined as
\begin{eqnarray*}
 \omega(\bs\xi, \bs\eta) = \sum_{i = 1}^n (\bs\xi_{p_i}\bs\eta_{q_i} - \bs\xi_{q_i}\bs\eta_{p_i}).
\end{eqnarray*}

A numerical integrator is called symplectic  if it preserves the two-form $d\bs p \wedge d\bs q$. More precisely, 
a differentiable map $g: U \rightarrow \mathbb{R}^{2n}$, where  $U\subset \mathbb{R}^{2n}$ is an open set,  is said to be symplectic if \[\omega(g'\bs\xi,g'\bs\eta) = \omega(\bs\xi, \bs\eta)\quad\text{for all}\ \bs\xi,\bs\eta\] where $g'$ denotes the Jacobian of $g$. 
A one-step numerical scheme is said to be symplectic if its one-step mapping is symplectic when applied to a smooth Hamiltonian system. As an example, for equation \eqref{ham}, the symplectic Euler scheme 
\begin{align*}
    & \bs p_{n+1} = \bs p_n - h\cdot\nabla_{\bs q}H(\bs p_n, \bs q_{n+1})\\
    &\bs q_{n+1} =\bs q_n + h\cdot\nabla_{\bs p}H(\bs p_n, \bs q_{n+1})
\end{align*}
is a one-step symplectic scheme of order 1; readers are referred to  \cite{hairer2006geometric} for more details on symplectic scheme. 

\section{Probability representations of WKB expansion}\label{sec:3}
This section presents the estimation of the quasi-potential  $V$ and prefactor $Z_0$ as a linear regression  based on the simulated data. In addition, a statistical criterion is introduced to evaluate the validity of the WKB approximation using residual analysis.

\subsection{Linear regression for  WKB approximation}\label{sub:linear_regression}
A crucial step in the DeepWKB framework concerns the  estimation of the quasi-potential $V$ and the prefactor $Z_0$ from the Monte Carlo data. To this end, the normalization factor $Q(\epsilon)$ needs to be determined. Although  an explicit expression is typically unavailable, its asymptotic order as $\epsilon\rightarrow 0$ can be identified under suitable conditions.

\begin{lem}
\label{lem:dim_Q}
  Let $V:\R^n\to[0,\infty)$ be a twice differentiable function such that  its zero-level set 
\[
  \mathcal{K}= \{ \bs x \in \mathbb{R}^{n} \,|\, V(\bs x) = 0\}
\]
is a $d$-dimensional compact smooth manifold embedded in $\R^n$. 
Assume that 
the Hessian $\nabla^2V(\bs x)$ 
is non-degenerate 
in the directions normal to $\mathcal K$ for all $\bs{x} \in \mathcal{K}$. Then there exists a  constant $C_0>0$ such that
\[
      \lim_{\epsilon \rightarrow 0}
      \frac{Q(\epsilon)}{\epsilon^{(n-d)/2}}=C_0.
\]
\end{lem}
\begin{proof}
Without loss of generality, assume $\mathcal K$ has a global coordinate chart; otherwise, the argument applies locally on a finite cover by coordinate charts.
For each $\bs x\in\mathcal K$ and $r>0,$ define the orthogonal neighborhood of $\bs{x}$ by
\begin{eqnarray*}
\mathcal N_r({\bs x})=\big\{\bs y\in\mathbb R^n:\|\bs y-\bs x\|< r\ \text{and}\ \langle\bs y-\bs x,v\rangle=0, \forall v\in T_{\bs x}\mathcal K\big\}, 
\end{eqnarray*}
where $\langle\cdot ,\cdot\rangle$ is the standard inner product on $\mathbb R^n$ and $\|\cdot\|$ is the 
induced norm. 

For sufficiently small $r>0$,  the union 
$\mathcal N_r:=\sqcup_{\bs x\in\mathcal K}\{\bs x\}\times\mathcal N_r({\bs x})$ forms a tubular neighborhood of $\mathcal K.$ 
In suitable coordinates, write 
$\mathcal N_r=\mathcal K\times B_r(\bs 0)$, where $B_r(\bs 0)\subset \mathbb R^{n-d}$ is the radius-$r$ ball centered at $\bs 0$. For $\bs x\in\mathcal N_r$, write $\bs x=(\bs x_1,\bs x_2)$, where 
 $\bs x_1=(x_i)_{i=1}^k$ 
 parametrizing  $\mathcal{K}$ and $\bs{x}_2 = (x_i)_{i=k+1}^n$ the transverse directions.
 Since each  $\bs x=(\bs x_1,\bs 0)\in\mathcal K$ is a local minimum of , it holds that 
\[\partial_{i} V(\bs x_1,\bs 0)=0, \quad \forall i\in\{1,...,n\}.
\]
Expanding $V$ in a Taylor series in the normal direction $\bs{x}_2$, one obtains 
\begin{align*}
V(\bs x_1,\bs x_2)&=\dfrac{1}{2}\sum\nolimits_{d+1\le i,j\le n}\partial^2_{ij}V(\bs x_1,\bs 0)x_ix_j
+\sum\nolimits_{d+1\le i,j\le n}h_{ij}(\bs x_1,\bs x_2)x_ix_j\\
&=\bs x^\top_2\Big(\dfrac{1}{2}\nabla^2V|_{\mathcal N_r(\bs x_1)}(\bs 0)+H(\bs x_1,\bs x_2)\Big)\bs x_2\\
&:=\bs x_2^\top A(\bs x_1,\bs x_2)\bs x_2, 
\end{align*}
where  
$h_{ij}(\bs x_1,\bs x_2)\to0$
as $\bs x_2\to\bs0$. Hence, 
\begin{eqnarray}\label{A_converge}
A(\bs x_1,\bs x_2)\to \dfrac{1}{2}\nabla^2V|_{\mathcal N_r(\bs x_1)}(\bs 0)
\end{eqnarray}
as $\bs x_2\to\bs0$. 

Since the probability density  $u_\epsilon$ concentrates near $\mathcal K$ as $\epsilon\to 0$, it holds that 
\begin{eqnarray}\label{u=1}
1=\lim_{\epsilon\to0}\int_{\mathcal N_r}u_\epsilon(\bs x)\mathrm{d}\bs x
\end{eqnarray}
By Fubini's theorem, 
\begin{eqnarray*}
\int_{\mathcal N_r}u_\epsilon(\bs x)\mathrm{d}\bs x&=&\int_{\mathcal N_r}Q(\epsilon)^{-1}Z_\epsilon(\bs x_1,\bs x_2)\exp\Big\{-\frac{1}{\epsilon}V(\bs x_1,\bs x_2)\Big\}\mathrm{d}\bs x_1 \mathrm{d}\bs x_2\\
&
=&\int_{\mathcal K}Q(\epsilon)^{-1}\int_{B_r(\bs 0)}Z_\epsilon(\bs x_1,\bs x_2)\exp\Big\{-\dfrac{1}{\epsilon}\bs x_2^\top A(\bs x_1,\bs x_2)\bs x_2\Big\}\mathrm{d}\bs x_2 \mathrm{d}\bs x_1.
\end{eqnarray*}
By the change of variables $\tilde{\bs x}_2=\bs x_2/{\sqrt{\epsilon}}$, the inner integral becomes
\begin{align*}
&\int_{B_r(\bs 0)}Z_\epsilon(\bs x_1,\bs x_2)\exp\Big\{-\dfrac{1}{\epsilon}\bs x_2^\top A(\bs x_1,\bs x_2)\bs x_2\Big\}\mathrm{d}\bs x_2\\
=&\ (\sqrt{\epsilon})^{n-d}\int_{B_{{r}/{\sqrt\epsilon}}(\bs 0)}Z_\epsilon(\bs x_1,\sqrt{\epsilon}\tilde{\bs x}_2)\exp\Big\{-\tilde{\bs x}_2^\top A(\bs x_1,\sqrt\epsilon\tilde{\bs x}_2)\tilde {\bs x}_2\Big\}\mathrm{d}\tilde {\bs x}_2.
\end{align*}
Therefore, 
\begin{eqnarray}\label{integration1}
&&\quad\int_{\mathcal N_r}u_\epsilon(\bs x)\mathrm{d}\bs x\\
&&=\int_{\mathcal  K}Q(\epsilon)^{-1}{\epsilon}^{(n-k)/2}\mathrm{d}\bs x_1\int_{B_{r/\sqrt\epsilon}(\bs 0)}Z_\epsilon(\bs x_1,\sqrt{\epsilon}\tilde{\bs x}_2)\exp\Big\{-\tilde{\bs x}_2^\top A(\bs x_1,\sqrt\epsilon\tilde{\bs x}_2)\tilde {\bs x}_2\Big\}\mathrm{d}\tilde {\bs x}_2.\nonumber 
 \end{eqnarray}

By assumption, 
the Hessian $\nabla^2 V|_{\mathcal N_r(\bs x_1)}(\bs 0)$ is non-degenerate, hence positive definite and invertible. 
Together with \eqref{A_converge}, 
the inner integration on the 
right-hand side of \eqref{integration1}
converges for each $\bs x\in\mathcal K$ to
\begin{align}
&
Z_0(\bs x_1,\bs 0)\int_{\R^{n-k}}\exp\Big\{-\dfrac{1}{2}\cdot\tilde{\bs x}_2^\top(\nabla^2V|_{\mathcal N_r(\bs x_1)}(\bs 0))\tilde{\bs x}_2\Big\}\mathrm{d}\tilde{\bs x}_2\nonumber\\
=&\ Z_0(\bs x_1,\bs 0)\sqrt{\det\big(2\pi(\nabla^2 V|_{\mathcal N_r(\bs x_1)}(\bs 0))^{-1}\big)}:=\ c(\bs x_1)\label{C}
\end{align}
 as $\epsilon\to0,$ where the  equality follows from the Gaussian integral formula. By  continuity of $Z_0(\bs x_1,\bs 0)$ and  compactness of $\mathcal K$, function $C(\bs x_1)$ is continuous and bounded away from zero. 

Combining \eqref{u=1},  \eqref{integration1} and \eqref{C}, it follows that the limit 
\[\lim\limits_{\epsilon\to0}Q(\epsilon)^{-1}\epsilon^{(n-d)/2}\]
exists and satisfies
\begin{eqnarray*}
1=\int_{\mathcal K}\Big(\lim_{\epsilon\to0}\dfrac{Q(\epsilon)}{\epsilon^{(n-d)/2}}\Big)^{-1}c(\bs x_1)\mathrm{d}\bs x_1,
\end{eqnarray*}
which implies  
\begin{eqnarray*}
\lim_{\epsilon\to0}\dfrac{Q(\epsilon)}{\epsilon^{(n-d)/2}}=\int_{\mathcal K}c(\bs x_1)\mathrm{d}\bs x_1:=C_0>0.
\end{eqnarray*}
\end{proof}

According to Lemma \ref{lem:dim_Q},  the constant $C_0^{-1}$ can be absorbed into the prefactor $Z_\epsilon(\bs x)$, so that  the normalization factor can be written as 
\[Q(\epsilon)=C(\epsilon)\epsilon^{(n-d)/2},\]
where $C(\epsilon)\to 1$ as $\epsilon\to0.$ Assuming $C(\epsilon)$ is at least $C^2$ in $\epsilon$, it admits the expansion \[C(\epsilon)=1+C_1\epsilon+O(\epsilon^2),\] and thus 
\begin{eqnarray}\label{Q}
Q(\epsilon)=(1+C_1\epsilon+O(\epsilon^2))\epsilon^{(n-d)/2}.
\end{eqnarray}

In this paper, the prefactor $Z_\epsilon$ is determined up to  first order in $\epsilon$. Hence, one writes 
\begin{equation}\label{WKB1}
 u_{\epsilon}(\bs x) = Q(\epsilon)^{-1}\big(Z_{0}(\bs x) + \epsilon Z_{1}(\bs x) + O( \epsilon^{2})\big)\exp \left ( -\frac{1}{\epsilon} V(\bs x) \right),
\end{equation}
where $Z_0$ agrees with \eqref{Z0}.
Taking the logarithm of both sides of \eqref{WKB1} and substituting the expansion \eqref{Q} for $Q(\epsilon),$ one obtains
\begin{equation}
  \label{WKBlog}
  \log u_{\epsilon}(\bs x) +\dfrac{n-d}{2}\log\epsilon =  \log Z_{0}(\bs x)  + \epsilon\Big(
  \frac{Z_{1}(\bs x)}{Z_{0}(\bs x)}-C_1\Big) - \frac{1}{\epsilon} V(\bs x)+
  O(\epsilon^{2}).
\end{equation}

Given the Monte Carlo  data of $u_\epsilon$ for multiple values of $\epsilon$, \eqref{WKBlog} can be employed to approximate $V$ and $Z_{0}$. To be specific,  perform  $K$ $(K>3)$ Monte Carlo simulations under 
the  noise strengths $\epsilon_{1}, \cdots, \epsilon_{K}$. For every $\epsilon_{i}$, estimate the invariant
probability density at each collocation point $\bm x^{*}$, denoting the  empirical density by $\hat{u}_{\epsilon_{i}}(\bs x^{*})$. Then solve the  least square optimization problem 
\begin{equation}
  \label{WKBlsq}
  {\bm\beta}^* = \arg \min_{{\bm \beta}} \|\mathcal{A} {\bm \beta} - \mathbf{\hat{u}} \|^2\,,
\end{equation}
where
\begin{eqnarray}\label{matrix_A}
 \mathcal{A} = 
\begin{pmatrix}
  -{1}/{\epsilon_{1}} & 1 & \epsilon_{1}\\
  \vdots &\vdots & \vdots\\
  -{1}/{\epsilon_{K}} & 1 & \epsilon_{K}
\end{pmatrix}\in \mathbb{R}^{K\times 3},
\end{eqnarray}
and 
\begin{eqnarray}\label{hat_u}
\mathbf{\hat{u}}=\big(\log \hat{u}_{\epsilon_{1}}(\bm x^{*}) +
  \frac{d-n}{2}\log \epsilon_{1},
  \, \cdots,  \, \log \hat{u}_{\epsilon_{K}}(
  \bm x^{*}) + \frac{d-n}{2}\log \epsilon_{K}\big)^{T}.
\end{eqnarray}

Let 
\begin{eqnarray}\label{beta}
 {\bm \beta}_0 =\Big(V(\bm x^{*}),\log Z_{0}(\bm x^{*}), 
\frac{Z_{1}(\bm x^{*})}{Z_{0}(\bm x^{*})}-C_1\Big)^{T}
\end{eqnarray}
denote the ground truth. 
The discussion above leads to 
\begin{eqnarray}\label{hat_u_eq}
 \mathbf{\hat{u}}  = \mathcal{A} {\bm \beta}_0 + \mathbf{s}_0 + \mathbf{s}_1,   
\end{eqnarray}
where $\mathbf{s}_0$ represents the sampling error from the Monte Carlo simulations, and $\mathbf{s}_1$ is an $O(\epsilon^2)$ term arising from the higher-order perturbation effects. 

It is well known that equation \eqref{WKBlsq} admits a least square solution (see \cite{bjorck2024numerical}), given by 
\begin{equation*}
  {\bm \beta}^{*}_0  = ( \mathcal{A}^{T} \mathcal{A})^{-1} \mathcal{A}^{T}
  \mathbf{\hat{u}}.
\end{equation*}
This yields a Monte Carlo-based  approximation of the quasi-potential
function $V(\bs x)$ and the leading-order term $Z_0(\bs x)$ in the prefactor. 

\subsection{Statistical validation of WKB approximation}\label{sec:WKB_statistical_validation}
Since a rigorous   justification  for the WKB approximation is not available in general, 
it is necessary to verify whether the invariant  density
function $u_{\epsilon}$ indeed admits the WKB form
\eqref{WKB}. In principle, this can be tested by
evaluating the residual term 
\begin{align*}
    \mathbf{r}_0:= \mathcal{A} {\bm \beta}^{*}_0 - \mathbf{\hat{u}}, 
\end{align*}
where $\mathcal{A}$, ${\bm \beta}^{*}_0$, and $\mathbf{\hat{u}}$ are defined in Section~\ref{sub:linear_regression}. 
Specifically, if the squared residual norm $\| \mathbf{r}_0 \|^{2}$, also known as the residual sum of squares (RSS), is unacceptably large at  many collocation points,
then the hypothesis that
$u_{\epsilon}({\bm x})$ admits a valid WKB expansion should be statistically rejected. 

To assess the reliability of the residual analysis, it is first necessary to estimate the uncertainty in the values of  $\log\hat{u}_{\epsilon_{i}}(
\bs x^{*})$ for each $i = 1, \cdots , K$, that is,  the entries of $\mathbf{s}_0$ in \eqref{hat_u}. As described in Section~\ref{sub:linear_regression}, $\hat{u}_{\epsilon_{i}}( {\bm x}^{*})$ is estimated by computing the
proportion of sample points falling inside the box centered at
${\bm x}^{*}$, with volume denoted by $h.$ Let $N$ denote the total number of samples,  and let $N_{0,i}$ be the number of those falling into the box in the $i$-th Monte Carlo simulation. Assuming that 
the samples are independently drawn from 
the invariant probability distribution, the count  $N_{0,i}$ follows a binomial distribution $B(N, u_{\epsilon_{i}}({\bm x}^{*})h)$. Applying  a normal approximation to the
binomial distribution yields
$$
  N_{0,i}\sim_{\rm{approx}}\mathcal{N}\big( Np_i ,Np_i(1-p_i)\big),
$$
where $p_i=u_{\epsilon_{i}}({\bm x}^{*})h$.
Since 
$\hat{u}_{\epsilon_{i}}({\bm x}^{*}) = N_{0,i}/(hN)$, one has $\hat p_i\approx N_{0,i}/N$, and consequently,
\begin{eqnarray*}
\hat{u}_{\epsilon_{i}}({\bm x}^{*})\sim_{\rm{approx}}\mathcal{N}\big(u_{\epsilon_i}(\bs x^*),N_{0,i}(1 -
N_{0,i}/N)h^{-2}N^{-2}\big).
\end{eqnarray*}
Hence, the estimation error
 in $\hat{u}_{\epsilon_{i}}({\bm x}^{*})$ can be approximated by a normal
distribution with zero mean and variance $N_{0,i}(1 -
N_{0,i}/N)h^{-2}N^{-2}$. Since this variance is typically small, the
corresponding error  in the logarithmic scale,
\[\delta_{i}({\bm x}^{*})=\log
\hat{u}_{\epsilon_{i}}({\bm x}^{*}) - \log
 u_{\epsilon_{i}}({\bm x}^{*})\] 
 can be approximated via a Taylor expansion. 
A straightforward calculation then gives 
\begin{align*}
\delta_{i}(\bm x^{*})\sim_{\rm{approx}} \mathcal{N}\big(0,(1-N_{0,i}/N)N_{0,i}^{-1}\big). \end{align*}
Therefore, it is reasonable to assume that 
$$
    \mathbf{s}_{0,i} = \Big(\sqrt{(1 - N_{0,i}/N)N_{0,i}^{-1}}\ \Big)Z_i,\quad \text{for}\ i=1,\ldots,K,
$$
where $Z_i$ are  independent and identically distributed standard normal random variables. 

Note that the entries of $\mathbf s_{0, i}$ have distinct variances. As a result, it is difficult to express the residual $\mathbf r_0$ in closed form. To address this issue, it is necessary to rescale the rows of the linear system  so that  all  entries of Monte Carlo error term become independent standard normal random variables. Define $D$ to be the diagonal matrix given by 
$$
D = \mbox{diag}\Big\{\sqrt{(1 - N_{0,1}/N)^{-1}N_{0,1}}\ ,\  \cdots,\  \sqrt{(1 - {N_{0,K}}/{N})^{-1}N_{0,K}}\Big\}.
$$
With this rescaling, the rescaled least square problem takes the form
\begin{eqnarray}\label{rescale_least_square}
{\bm\beta}^* = \arg \min_{\bm{\beta}} \|D\mathcal{A} \bm{\beta} - D\mathbf{\hat{u}} \|^2,
\end{eqnarray}
and the corresponding residual is given by
\begin{eqnarray}\label{rescaled_r}
\mathbf{r}= D \mathcal{A} \mathbf{\bs\beta}^{*} - D
    \mathbf{\hat{u}}.
\end{eqnarray}

\begin{prop}\label{prop:CI}
Let $\hat{\mathbf{u}}$ be defined as in \eqref{hat_u}, and let $\mathbf{r}$ be the rescaled residual given by \eqref{rescaled_r}. Then
$$
 \mathbf{r} = \mathbf{r}_0 + \mathbf{r}_1, 
$$
where $\|\mathbf{r}_0\|^2\sim\chi^2(K-3),$ and 
\[
    \|\mathbf{r}_1\| \leq \|(I - P_A)D\|O(\epsilon^2)
\]
with the projection matrix 
\[P_A = D\mathcal{A}( \mathcal{A}^{T} D^2\mathcal{A})^{-1}\mathcal{A}^TD.\]
\end{prop}

\begin{proof}
For notational convenience, let $\mathcal{A}_1:= D \mathcal{A}$ and $\mathbf{\hat{u}_1}:=D\mathbf{\hat{u}}.$ The least square solution to \eqref{rescale_least_square} is
given by
$$
    {\bm \beta}^* = ( \mathcal{A}_1^{T} \mathcal{A}_1)^{-1} \mathcal{A}^{T}_1
 D\mathbf{\hat{u}}.
$$
Hence, the fitted value is 
$$
\mathcal{A}_1{\bm \beta}^* = \mathcal{A}_1( \mathcal{A}_1^{T} \mathcal{A}_1)^{-1} \mathcal{A}^{T}_1
\mathbf{\hat{u}_1} := P_A\mathbf{\hat{u}_1},
$$
where $P_A$ is the orthogonal projection onto the column space of $\mathcal{A}_1$. 
Using the decomposition of $\mathbf{\hat{u}}$ in \eqref{hat_u_eq},  the scaled residual becomes 
$$
\mathbf{r} =(P_A-I)\mathbf{\hat{u}_1} =(P_A-I)( \mathcal{A}_1 {\bm \beta}_0 + \mathbf{Z} + D \mathbf{s}_1 )\,,
$$
where the entries of $\mathbf{Z}$ are independent standard normal random variables. Since $( P_A-I)\mathcal{A}_1  = 0$, it follows that
$$
\mathbf{r} = (P_A-I)\mathbf{Z} + (P_A-I) D \mathbf{s}_1 := \mathbf{r}_0 + \mathbf{r}_1.
$$
Note that $(P_A-I)$ is a projection matrix of rank $(K - 3)$. Therefore, there exists an orthonormal matrix $U \in \mathbb{R}^{K\times (K - 3)}$ such that \[{(P_A-I)}\mathbf{Z} = UU^T \mathbf{Z}.\] Since $\mathbf{Z}\sim\mathcal N(0,I_K)$, it follows that $U^T\mathbf{Z}\sim\mathcal N(0,I_{K-3})$, and hence 
\[\|\mathbf{r}_0\|^2 = \|U^T \mathbf{Z}\|^2\sim\chi^2(K-3).\] 
Finally, since $\mathbf{s}_1=O(\epsilon^2)$, it follows that
$$
    \|\mathbf{r}_1 \| \leq \| (P_A-I)D\| O(\epsilon^2).
$$
\end{proof}

Proposition \ref{prop:CI} establishes a statistical criterion for assessing the  WKB approximation
 of invariant distribution $u_\epsilon$. 
When $\epsilon$ is sufficiently small, the $\epsilon^2$-terms becomes negligible, and the rescaled RSS, $\| \mathbf{r} \|^{2}$, can be evaluated at all collocation points. If the computed rescaled RSS follows a $\chi^{2}$ distribution with $K-3$ degrees of freedom, then
$u_{\epsilon}$ can be regarded as  statistically consistent with  the WKB form. 

\begin{rmk}
{\rm 
In practice, factors other than the $\epsilon^2$-term can also affect the probability distribution of RSS. First,  Monte Carlo samples are generated from a long trajectory of \eqref{SDE},  introducing temporal correlations  
rather than  truly independent and identically distributed. To reduce these correlations, samples should be taken with sufficiently large time intervals, which lowers the effective sample size below $N_{0}$ and increases the relative error 
beyond that predicted by the binomial distribution, 
resulting in a higher observed RSS. Second, when the probability density varies significantly within a bin, the  estimation becomes less accurate. This occurs because the collocation point is assumed to lie at the bin center, whereas samples may cluster near the edges or corners. Therefore, 
for reliable statistical validation of the WKB approximation, the noise magnitude $\epsilon$ should be significantly larger than the bin size. 
}
\end{rmk}

The following Algorithm \ref{alg:WKB} provides  a statistical procedure to decide whether the WKB approximation holds for the invariant distribution of SDE \eqref{SDE}.

\begin{algorithm}
\caption{Statistical validation of WKB approximation}
\label{alg:WKB}
\begin{algorithmic}
\Require Stochastic differential equation \eqref{SDE}
\Require Collocation points ${\bm x}^* _1, \cdots, {\bm x}^* _M$
\Ensure Whether WKB approximation holds (Y or N)
\State Choose $K(K>3)$ small numbers $\epsilon_1, \cdots, \epsilon_K$
\For{ $i = 1$ to $K$}
\State Let the strength of noise in equation \eqref{SDE} be $\epsilon_i$
\State Run Monte Carlo simulations of \eqref{SDE} to estimate the probability densities at collocation points $j = 1, \cdots, M$, denoted by  $\hat{u}_{\epsilon_i}({\bm x}^* _j) = N_{0, j}/(hN).$
\EndFor
\For{ $j = 1$ to $M$}
\State Build the rescaling matrix $D$ with $D_{ii} = \sqrt{(1 - {N_{0,i}}/{N})^{-1}N_{0,i}}$
\State Build matrix $\mathcal{A}$ as in \eqref{matrix_A}
\State Solve the normal equation $D \mathcal{A} \mathbf{x} = D \hat{\mathbf{u}}$
\State Calculate RSS $\|\mathbf{r}\|^2$
\If {$\|\mathbf{r}\|^2$ statistically satisfies the $\chi^2$-distribution with $(K-3)$ degrees of freedom}
\State WKB approximation holds
\Else
\State WKB approximation does not hold
\EndIf
\EndFor
\end{algorithmic}
\end{algorithm}

\section{Deep neural network solver for quasi-potential function}\label{sec:V}
\subsection{Construction of training set and loss functions}
The quasi-potential function $V$ is approximated by  an artificial  neural network, denoted by $V_{\bm \theta}({\bm x})$, where ${\bm \theta}$ denotes the set of trainable network parameters. The training set consists of a collection of collocation points, which serves as the input set, together with the corresponding approximations of $V$ at these points, serving as the output set.  To captures the properties of $V$ in different regions of the domain, the  training set is constructed in the following three components. 
\begin{enumerate}
    \item {Training set for the attractor $\mathcal{K}$}.
    \\
The input set for the first part of the training set, denoted by $\mathbf{X}_1:= \{ {\bm x}^a_1, \dots, {\bm x}^a_{M_1} \}$, consists of points sampled from the global attractor $\mathcal{K}$ on which the quasi-potential $V$ vanishes. These points are obtained by numerically integrating the deterministic system \eqref{ODE}  until the trajectory converges  sufficiently close to the attractor. A subset of points  $\{{\bm x}^a_i\}$ is then randomly sampled from the trajectory and assigned zero output values, yielding the  training set $(\mathbf{X}_1, \mathbf{0})$. This component ensures that the neural network accurately captures the vanishing property of $V$ on the attractor. 

    \vspace{1mm}

    \item {Training set for WKB approximations}. \\  
Since the invariant distribution of \eqref{SDE} is primarily concentrated near the attractor, a large number of collocation points is required in that region. To this end, the collocation points for the second part of the training set, $\mathbf{X}_2 := \{ {\bm x}^* _1, \dots, {\bm x}^* _{M_2} \}$,  
 consists of two subsets:  one obtained by sampling from long trajectories of the stochastic system \eqref{SDE}, yielding points concentrated near the attractor; and the other generated by uniformly sampling over the entire numerical domain. 
The full set of collocation points $\{\bs x_i^*\}$ are generated  
using the algorithm developed in \cite{zhai2022deep}, wherein half of the collocation points are sampled from a long trajectory of equation \eqref{SDE}, while the other half are uniformly distributed in the domain. The corresponding output values, denoted by  $\hat{V}({\bm x}^* _i)$, are computed using the linear regression procedure introduced in Section \ref{sub:linear_regression}, 
based on the WKB expansions of the invariant  density. The resulting training set $({\mathbf X}_2, \hat{V}({\mathbf X}_2))$ guides the neural network to approximate the WKB solution of the stationary Fokker-Planck equation \eqref{stationary_FPE}. 

  \vspace{1mm}

    \item {Training set for operator residual}.
    \\
    The third part of the training set  is used to control the residual of the Hamilton-Jacobi operator. Denote the input set by $\mathbf{X}_3:$ $= \{ {\bm x}^h_1, \dots, {\bm x}^h_{M_3}\}$, where the collocation points 
    $\{\bs x_i^h\}$ are  generated in a manner similar to those in $\mathbf{X}_2$, but in significantly larger quantity, as no output values are required for these points.    
    Since the quasi-potential $V$ satisfies the Hamilton-Jacobi equation \eqref{HJE}, the neural network is  trained to minimize the residual of the associated operator, defined at each collocation point by    
    \[
  \mathcal{H}V_{\bm \theta}({\bm x}^h_i) = \sum_{i=1}^{n}f^{i} \partial_i V_{\bm \theta}({\bm x}^h_i) +
    \frac{1}{2} \sum_{i,j = 1}^{n} a^{ij} \partial_iV_{\bm \theta}({\bm x}^h_i)
    \partial_j V_{\bm \theta}({\bm x}^h_i).
    \]
    
    \end{enumerate}
\medskip

Corresponding to the three training sets above, the loss function for  the neural network for approximation of $V$ consists of the following three components:
\begin{enumerate}
    \item Loss function corresponding to  ${\mathbf X}_1$: 
    $$
        L_1^V = \frac{1}{M_1}\sum_{i = 1}^{M_1} V_{\bm \theta}({\bm x}^a_i)^2;
    $$
    
    \item Loss function corresponding to ${\mathbf X}_2$: 
    $$
        L_2^V = \frac{1}{M_2} \sum_{i = 1}^{M_2} ( V_{\bm \theta}({\bm x}^* _i) - \hat{V}({\bm x}^* _i))^2;
    $$

    \item Loss function corresponding to ${\mathbf X}_3:$ 
\[L_3^V = \frac{1}{M_3} \sum_{i = 1}^{M_3} ( \mathcal{H}V_{\bm \theta}({\bm x}^h_i) )^2.
\]
\end{enumerate} 

The terms $L_1^V$ and $L_2^V$ guide the neural network in fitting the quasi-potential function to its  reference values at the collocation points, whereas     
$L_3^V$ enforces the Hamilton-Jacobi equation satisfied by $V$ by penalizing its residual. 

\subsection{Strategies in neural network training}
 The quasi-potential  $V$ satisfies the Hamilton-Jacobi equation \eqref{HJE}, which, however,  admits a trivial solution $V = 0$. In practice, training the neural network only using the Monte Carlo simulation data  often results in  convergence to this trivial solution. Moreover, since $V$ appears in  the exponent of the WKB expansion, even a small error in its approximation near the attractor can lead to  significant inaccuracies in the resulting approximation of $u_\epsilon$. To mitigate these issues, the following strategies are adopted to
improve  training performance.  

\begin{enumerate}
    \item {Collocation points on the attractor:}    
    As described in the first part of the training set construction, a subset of collocation points $\mathbf{X}_1$ is sampled from the attractor. A corresponding loss term $L_1^V$ is introduced
    to enforce that the learned quasi-potential $V$ remains close to zero on the attractor $\mathcal{K}$.    
    This acts as a strong regularization mechanism. 

     \vspace{1mm}
     
    \item Penalty term: To discourage the quasi-potential from taking negative values, penalty terms of the form $\max\{ -V, 0\}$ are  incorporated into the loss functions $L_1^V$ and $L_2^V$ during the training.        
         
\vspace{1mm}
    
    \item Artificial values far away from the attractor:   Most 
    collocation points obtained from  Monte Carlo simulations are  concentrated near the attractor, which may lead to poor neural network performance in regions  farther away and increase the risk of convergence to the trivial solution. To address this, additional collocation points located far  from the attractor are randomly sampled, as described in the construction of the set $\mathbf {X}_2.$ However, due to the rapid decay of the Monte Carlo accuracy
    in low-density regions where $u_\epsilon$ is small, 
    reliable outputs are  only obtained when the quasi-potential $V(\bm x^*)$ is below a certain threshold. To compensate for the lack of reliable data in these regions, 
    an artificial value $C$ is assigned;  that is,  the estimated quasi-potential is set to $C$ when the true quasi-potential exceeds this threshold.    
    
     \vspace{1mm}
    
    \item Alternating Adam training method: Since the loss components $L_1^V, L_2^V$, and $L_3^V$ may differ significantly in scale, the Adam optimizer is applied in an  alternating fashion to each component during training, following the strategy proposed in \cite{zhai2022deep}.    
    This approach exploits  the near scale-invariance of the Adam optimizer (see \cite{kingma2014adam}) and is  straightforward to implement in practice. 

     \vspace{1mm}
         
    \item Fine-tuning procedure: In the final stage of training,   once the loss has approximately converged, the neural network is further trained for a few additional epochs using only the loss components $L_1^V$ and $L_3^V$, with a significantly reduced learning rate. This  fine-tuning step helps suppress the residual noise from the Monte Carlo simulation data and enhances the overall accuracy of the learned solution. 
\end{enumerate}

\subsection{Training set expansion}
As discussed in the previous subsection,  Monte Carlo simulation provides reliable approximations only in regions where the invariant density function is not too small (typically greater than $10^{-6}$), and hence collocation points with artificially assigned output values are introduced.  However, this restricts  further exploration of the quasi-potential landscape. To overcome this limitation, the training set is expanded by adding collocation points  located far from the attractor (i.e., in  low-density regions), together with their corresponding quasi-potential values. These additional data are incorporated into the  loss component $L_2^V$ to enhance the accuracy of the neural network training. 

We remark that after the training set expansion, it is crucial to continue training from the previously learned neural network parameters, rather than restarting training from scratch. This approach is consistent with the principle of curriculum learning in machine learning \cite{bengio2009curriculum}.
In addition, the artificial values initially assigned to collocation points far from the attractor should be updated to reflect the evolving approximation of the quasi-potential during training.

In this paper, two approaches are employed to expand the training set:  the method of characteristics and the minimal action method. 

\vspace{1mm}

{\bf Method of characteristics.} This approach leverages the fact that the Hamilton-Jacobi equation \eqref{HJE} can be solved along characteristic curves. Let $\bs p = \nabla V$, and the Hamiltonian associated with the equation \eqref{HJE} is given by
\[H = f^T\bs p + \frac{1}{2} \bs p^T A\bs p.\] It can be readily  to verify that $\dot H=0$ along the characteristic system 
\begin{eqnarray}\label{char}
\left\{
\begin{aligned}  
    \dot{\bs x} &= \nabla_{\bs p} H = f + A\bs p\\
    \dot{\bs p} &= -\nabla_{\bs x} H = -(\nabla f)^T\bs p
\end{aligned}\right.
\end{eqnarray} 
where $\nabla f \in \mathbb{R}^{n\times n}$ denotes the Jacobian of the vector field $f$. Note that \eqref{char} defines  a Hamiltonian system with  conserved quantity $H$. A direct calculation then yields
\begin{equation}
    \label{charV}
    \dot{V} = {\nabla_{\bs x} V}\cdot \dot{\bs x} = \bs p \cdot \dot{\bs x} = \frac{1}{2}\bs p^T A \bs p.
\end{equation}

By solving the characteristic system \eqref{char}, the equation \eqref{charV} facilitates the expansion of the training set to a larger region as follows. First,  collocation points are sampled from regions where the quasi-potential $V$ has  been determined, and the corresponding derivatives $p = \partial V / \partial x_i$ are obtained from the  neural network trained in the previous subsection. 
Next, integrating  \eqref{char} starting from each of these points extends the solution $V$ into   regions beyond the 
reach of Monte Carlo simulations. The newly generated collocation points are  then added to the training set $\mathbf{X}_2$, enabling the removal of the original collocations points with artificial assigned values located far from the attractor. This procedure can be repeated iteratively to further expand the training set if necessary. Since  \eqref{char} defines a Hamiltonian system, a symplectic integrator (see Section \ref{sec:symplectic scheme}) should be used to avoid the  rapid accumulation of numerical errors. 

\medskip

 {\bf Minimum action method.} This approach exploits the fact that the Freidlin-Wentzell quasi-potential $V$  satisfies the minimum action principle. More precisely, $V$ admits the  variational characterization 
$$
    V(\bs x) = \inf_{T > 0}\inf_{\substack{\phi\in C[0, T],\\\phi(0) \in \mathcal K, \phi(T) = \bs x}}\left \{ \int_0^T \frac{1}{2} (\dot{\phi} - f)^T A^{-1} (\dot{\phi} - f) \mathrm{d}s\right\}, 
$$
where the integral is known as the action functional of the path $\phi$. Similar to the method of characteristics, once the minimizing path $\phi$ is found for a given $\bs x \in \mathbb{R}^n$, the value $V(\phi(t))$ can be evaluated for all $t \in [0 , T]$. 

If the attractor is  a stable equilibrium, then $V(\bs x)$ can be numerically computed by applying the geometric minimum action method (gMAM). However, for a general attractor, gMAM becomes less  effective because the  minimizing path $\phi$ is often infinitely long. In such cases, it is more practical to  start from an energy surface defined by 
\[\mathcal S := \{\bs x \in \mathbb{R}^n \,|\, V(\bs x) = C \}.\] It is straightforward to verify that for any $\bs x$ with $V(\bs x) > C$, the quasi-potential admits  the  variational representation 
$$
    V(\bs x) = C + \inf_{T > 0}\inf_{\substack{\phi\in C[0, T],\\ \phi(0) \in \mathcal S, \phi(T)=\bs x}}\left \{ \int_0^T \frac{1}{2} (\dot{\phi} - f)^T A^{-1} (\dot{\phi} - f) \mathrm{d}s\right\}.
$$
That is,  the infimum is taken over all paths starting from the  energy surface. In practice, a collection of initial points $\bs y_1, \cdots, \bs y_p$ are sampled from $\mathcal S$, and the minimum action is approximated by minimizing over the corresponding action functionals from each $\bs y_i$ to $\bs x$. 

Although the minimum action method incurs higher computational cost, it serves as a valuable complement to the method of characteristics,  which lacks explicit control over the directions of the characteristic curves and may consequently fail to cover certain regions of the training domain.  In such cases,  additional collocation points are sampled from the uncovered regions, and gMAM is applied to compute  minimizing paths to these  points, thereby enhancing the coverage and quality of the training samples. 

\section{Deep neural network solver for prefactor function}\label{sec:Z}
\subsection{Training set and loss functions}
In the small noise regime, learning the leading-order term $Z_0(\bs x)$  is sufficient to accurately estimate the prefactor  $Z_\epsilon(\bs x)$. 
Similar to the quasi-potential training, the training set for $Z_0(\bs x)$ consists of three parts, with the corresponding input sets given by 
\begin{enumerate}
\item  A subset of collocation points $\mathbf{Y}_1:= \{\bs y^a_1, \cdots, \bs y^a_{M_1} \}$ located on the attractor $\mathcal{K}$;
 
\item A subset of collocation points
$\mathbf{Y}_2 := \{\bs y^*_1, \cdots, \bs y^*_{M_2} \}$ generated by Monte Carlo simulation;  

\item A subset of collocation points $\mathbf{Y}_3:=\{\bs y^l_1, \cdots, \bs y^l_{M_3} \}$ used to evaluate the residual of the PDE operator $\mathcal{L}^0$, as defined in \eqref{transZ}.
\end{enumerate}

The sampling strategy for $\mathbf{Y}_2$ and $\mathbf{Y}_3$ is the same as that used for training $V$: half of the collocation points are extracted from a long trajectory of the stochastic dynamics \eqref{SDE}, while the other half are sampled uniformly over the entire numerical domain.  

The output values $\hat{Z}_0(\bs y^*_i)$  for the training set $\mathbf{Y}_2$ are obtained from the least square solution of the optimization problem \eqref{WKBlsq}, given by the  exponential of the second entry in \eqref{beta}. 
For 
the output values on the training set $\mathbf{Y}_1$, noting that $V=0$ on the attractor $\mathcal K,$ the WKB expansion of the invariant density $u_\epsilon$ on $\mathcal{K}$ reduces to
$$
    u_\epsilon(\bs x) = Q(\epsilon)^{-1}(Z_0(\bs x) + \epsilon Z_1(\bs x) + O(\epsilon^2)
$$
where, by Lemma \ref{lem:dim_Q}, $Q(\epsilon) = \epsilon^{(n-d)/2}$ for sufficiently small $\epsilon>0$. 
A linear regression of $\hat{u}_{\epsilon_i}(\bs y_i^a)$ over multiple  noise magnitudes $\epsilon_1,\ldots,\epsilon_K$ yields $\hat Z_0(\bs y^a_i)$ via  linear extrapolation to $\epsilon = 0$. 
This provides a better estimate of $Z_0$, as it avoids the uncertainties of $V$ and the errors from the Taylor expansion.

Analogous to the quasi-potential $V$, the loss function for $Z_0$ consists of the following three components:
\begin{enumerate}
    \item Loss function corresponding to $\mathbf{Y}_1$:
    $$
        L^Z_1 = \frac{1}{M_1}\sum_{i = 1}^{M_1} (Z_0(\bs y^a_i) - \hat{Z}_0(\bs y^a_i) )^2;
    $$

    \item Loss function corresponding to $\mathbf{Y}_2:$
    $$
        L^Z_2 = \frac{1}{M_2} \sum_{i = 1}^{M_2} ( Z(\bs y^*_i) - \hat{Z}(\bs y^*_i))^2; 
    $$
   
    \item Loss function corresponding  to $\mathbf{Y}_3:$
    $$
        L^Z_3 = \frac{1}{M_3} \sum_{i = 1}^{M_3} ( \mathcal{L}^0 Z(\bs y^l_i) )^2.
    $$
\end{enumerate} 
 
 It is important to note that  the operator $\mathcal{L}^0$ involves the quasi-potential $V$ 
 and utilizes the previously trained  $V_{\bm \theta}$ from Section \ref{sec:V}. However, the Hamilton-Jacobi equation for $V$ is invariant under constant scaling. As a result, the neural network  output after fine-tuning is typically $\alpha V$ rather than $V$ itself, where the scaling factor $\alpha$ is usually slightly less than $1$. The factor $\alpha$ can be estimated by comparing the trained values  $V_{\bm \theta}({\bs x}_i^*)$ with the corresponding reference values $\hat V(\bm x_i^*)$ at collocation points and then averaging the  ratios. 

\subsection{Training set expansion}\label{sec:training set expand}Similar to the case of quasi-potential, the training set for $Z_0(\bs x)$ can also be expanded by solving $Z_0$ along its characteristics curves. Recall that $Z_0$ satisfies the transport equation 
\begin{eqnarray*}
0 &=&\sum_{i = 1}^{n} \left( f^{i} + \sum_{j = 1}^{n} a^{ij}
     \partial_j V
     \right ) \partial_iZ_{0}\\
     &\ &+ \left ( \sum_{i = 1}^{n} \partial_if^{i} + \frac{1}{2}
     \sum_{i,j = 1}^{n} a^{ij} \partial^{2}_{ij} V + \sum_{i,j =
     1}^{n}\partial_{i} a^{ij} \partial_jV\right )Z_{0}\\
     &:=& \sum_{i = 1}^n b^i \partial_iZ_0 + c Z_0. 
\end{eqnarray*}

Let $\omega(s)$ be the characteristic curve defined by 
$$
\dot{\omega}(s) = b(\omega(s)),
$$
where $b=(b^i).$
Along the curve $\omega(s),$ the function $Z_0$ satisfies the ordinary differential equation 
$$
\frac{\mathrm{d}}{\mathrm{ds}} Z_0(\omega(s)) + c Z_0(\omega(s)) = 0.
$$
This determines $Z_0$ on each characteristic curve. Since $\omega(s)$ is also the characteristic curve of $V$, i.e., the $x$-component of equation \eqref{char}, one only needs to evaluate $c$ along $\omega(s)$ and  compute
$$
Z_0(\omega(s)) = Z_0(\omega(0))\exp\Big\{-\int_0^s c(\omega(r))\mathrm{d}r\Big\}.
$$
This extends the training set of $Z_0$ from a neighborhood of the attractor to a large region, thereby  improving the quality of the training. 

\subsection{Neural network training}\label{sec:NN_Z}
The neural network training for $Z_0$ follows a similar procedure as that for the quasi-potential $V$. The Adam optimizer is applied alternatively to  the three loss components $L^Z_1,L^Z_2$, and $L^Z_3.$ To enforce 
non-negativity of $Z_0$, the penalty term 
$\max\{ -Z_0, 0\}$ is incorporated into the losses $L^Z_1$ and $L^Z_2$.  The training is finally fine-tuned using only $L^Z_1$ and $L^Z_3$.  

Compared to the quasi-potential $V$, the training accuracy for the prefactor $Z_0$ deteriorates when farther from the attractor. This is due to the fact that the ground truth data of $Z_0$, obtained from the Monte Carlo simulations, is available only in the vicinity of the attractor. Reference data for more distant regions are generated by expanding the training set, where both the loss $L^Z_3$ and the characteristic curves of $Z_0$ depend on the first and second derivatives of $V$. These derivatives are generally less accurate than the values of $V$ itself, resulting in reduced accuracy of $Z_0$ in those regions. 
Since this paper focuses on the regime $0<\epsilon \ll 1$, where  the Fokker-Planck equation becomes highly singular, the effect of inaccuracies in $Z_0(\bm x)$ is negligible
in regions far from the attractor, as the leading  term $\exp{-V/\epsilon^2}$ is already exponentially small.

\section{Numerical examples}\label{sec:example}
This section applies the DeepWKB method 
to compute the invariant distributions of several stochastic systems by approximating their WKB expansions, where the corresponding deterministic dynamics exhibit non-trivial attractors. In subsection \ref{sec:WKB_verify}, the validity of the WKB expansions are  statistically verified.

In all the numerical examples, a six-layer feedforward neural network ${\rm NN}:\mathbb{R}^n \to \mathbb{R}$ is employed, with  layer widths 
\[n\to 32 \to 256 \to 256 \to 256 \to 64 \to 16 \to 1,\]
where $n$ denotes the dimension of the system. The network uses sigmoid activation functions and incorporates $L_2$ weight regularization with penalty parameter $\lambda = 10^{-3}$. 

\subsection{Stochastic Van der Pol oscillator}
The first example considered is the stochastic Van der Pol oscillator
\begin{eqnarray}\label{eq-VDP}
\left\{\begin{aligned}
    &\mathrm{d}X_t = \mu(X_t - \frac{1}{3}X_t^2 - Y_t) \mathrm{d}t + \sqrt{\epsilon} \mathrm{d}B^{(1)}_t\\
    &\mathrm{d}Y_t = \frac{1}{\mu}X_t \mathrm{d}t + \sqrt{\epsilon} \mathrm{d}B^{(2)}_t,
\end{aligned}
\right.
\end{eqnarray}
where  $\mu$ denotes the scaling parameter of  the nonlinear damping. It is well-known that  the deterministic part of \eqref{eq-VDP} admits a stable limit cycle.

In the simulations, $\sqrt\varepsilon$ takes values $0.08, 0.09, 0.1, 0.11, 0.12, 0.14, 0.16, 0.18, 0.2$, and $0.24$. For each $\epsilon$, samples are collected on a $1024\times 1024$ bin grid over the domain $[-3,3]\times[-3,3]$ 
by simulating $10$ parallel trajectories of equation \eqref{eq-VDP} up to time $t = 10^7$.  The deterministic system is integrated
up to time $t = 50000$, yielding  3551 collocation points on the limit cycle. In addition, 20000 collocation points without output values are generated to minimize the loss $L^V_3$  using the algorithm proposed in \cite{zhai2022deep}. The training set is illustrated in the top-left panel of Figure \ref{vdp-fig1}. 

The network is trained for 200 epochs are implemented using a batch size of 128. 
The training minimizes the composite loss
\[
L^V(\bs\theta):=\; L_1^V(\bs\theta) \;+\; L_2^V(\bs\theta) \;+\; L_3^V(\bs\theta),
\]
where
\begin{align*}
L_1^V(\bs\theta)&= \mathbb{E}_{\bs x\in\mathbf{X}_1}\bigl[V_{\bs\theta}(\bs x)^2 + \max\{0,\,-V_{\bs\theta}(\bs x)\}\bigr],\\
L_2^V(\bs\theta) &= \mathbb{E}_{\bs x\in \mathbf{X}_2}\bigl[(V_{\bs\theta}(\bs x)-\hat V(\bs x))^2 + \max\{0,\,-V_{\bs\theta}(\bs x)\}\bigr],\\
L_3^V(\bs\theta) &= \mathbb{E}_{\bs x\in \mathbf{X}_3}\bigl[b(\bs x)\!\cdot\!\nabla V_{\bs\theta}(\bs x) + \tfrac12\|\nabla V_{\bs\theta}(\bs x)\|^2\bigr]^2.
\end{align*}
Three Adam optimizers with learning rates $2\times10^{-4}$, $5\times10^{-4}$, and $1\times10^{-4}$  are applied alternately to the losses $L^V_1$, $L^V_2$, and $L^V_3$, respectively. 
The training result is shown in the top-middle panel of Figure \ref{vdp-fig1}.

After training the quasi-potential $V,$ another $171$ collocation points are random selected near the level set $V = 0.05$, and the partial derivatives of $V_{\bs\theta}$ are evaluated at these points. The symplectic Euler method with a step size $10^{-4}$  is then employed to  integrate  the characteristic curves of the Hamilton-Jacobi equation up to the level set $V = 0.4$
The resulting characteristic curves are displayed in the top-right panel of Figure \ref{vdp-fig1}, where $20$ collocation points are uniformly sampled along each curve. 
The final learned quasi-potential $V_{\bm \theta}$ is shown in the bottom-left panel of Figure \ref{vdp-fig1}. 

Finally, using the trained $V_{\bs \theta}$, a second neural network with the same architecture is trained to approximate the prefactor $Z_0$. 
The network architecture is identical to that used for learning $V.$ The output values of $Z_0$  at 
$3000$ randomly sampled  collocation points  near the attractor 
are obtained via linear regression. 
To improve training quality, $Z_0$ is further evaluated along  the characteristic curves using the method described in Section \ref{sec:training set expand}. From these characteristics, an additional set of $3000$ collocation points is randomly selected. The training result for $Z_0$ is shown in the  bottom-middle panel of Figure \ref{vdp-fig1}. 

As discussed in Section \ref{sec:NN_Z}, the training accuracy of $Z_0$ deteriorates in regions far from the attractor. 
This is because the training of function $Z_0$ relies on the second derivatives of $V_{\bm \theta}$, whereas  the loss function $L_3^V$ only regularizes its first derivatives. 
However, 
since the leading term in the WKB approximation,  $\exp\{-V/\epsilon\}$, becomes exponentially small away from the attractor, accurate evaluation of $Z_0$ is essential only near the attractor. 
The bottom-right panel of  Figure \ref{vdp-fig1} illustrates the WKB approximation $Z_0 \exp\{-V/\epsilon\}$  with $\epsilon = 0.0225$, showing good agreement with the invariant probability density $u_\epsilon$. 

\begin{figure}[htbp]
    \centering
    \includegraphics[width=\linewidth]{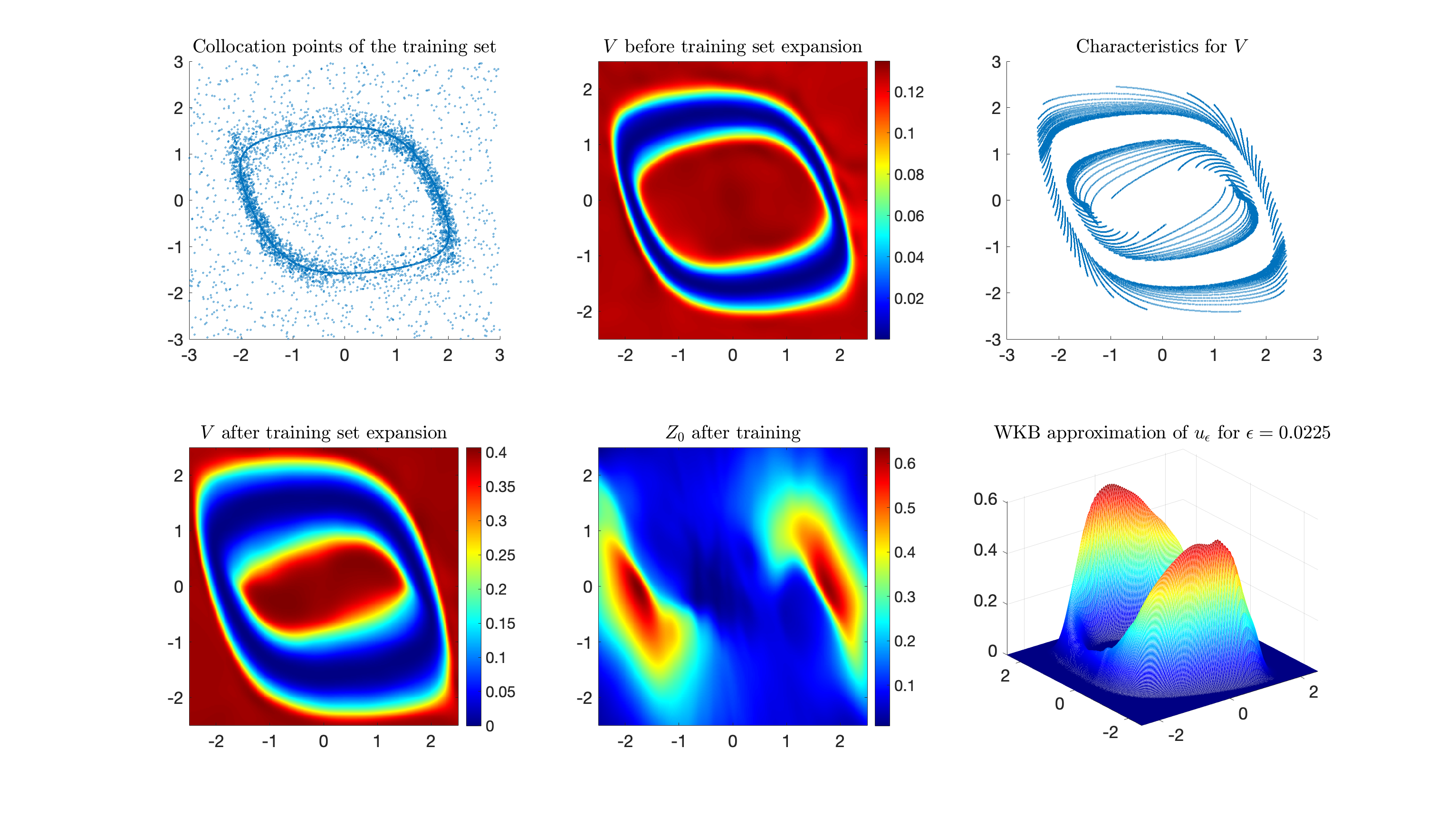}
    \caption{Top Left: Collocation points of training sets $\mathbf{X}_1$ and $\mathbf{X}_2$. Top Middle: $V_{\boldsymbol\theta}$ trained before training set expansion. Top Right: Expansion of the training set using method of characteristics. Bottom Left: $V_{\boldsymbol\theta}$ trained after training set expansion. Bottom Middle: Neural network approximation of $Z_0$. Bottom Right: WKB approximation $u_\epsilon \approx Q(\epsilon)^{-1}Z_0 \exp\{-V/\epsilon\}$ for $\epsilon = 0.0225$. }
    \label{vdp-fig1}
\end{figure}

\subsection{Stochastic dynamics of figure-eight attractor} 
This subsection investigates the WKB expansion of invariant distributions for stochastic perturbations of a figure-eight attractor, given by 
\begin{eqnarray} \label{eq-f8}
\left\{
\begin{aligned}   
    \mathrm{d}X &= \left ({\partial_y H} - \mu H{\partial_x H} \right ) \mathrm{d}t + \sqrt{\epsilon} \mathrm{d} B^{(1)}_t\\
    \mathrm{d}Y &= \left( -{\partial_x H} - \mu H {\partial_y H} \right)\mathrm{d}t +  \sqrt{\epsilon} \mathrm{d} B^{(2)}_t \,,
\end{aligned}
\right.
\end{eqnarray}
where the damping strength $\mu$ is set to $0.5$. The deterministic part of \eqref{eq-f8} is a dissipative perturbation of a Hamiltonian system with  Hamiltonian 
\[
 H = \frac{1}{2} Y^2 + \frac{1}{12}X^4 - \frac{1}{2}X^2.
\]
The corresponding vector field decomposes into the sum of two orthogonal components $(-\mu H H_x, -\mu H H_y)$ and $(H_y, -H_x)$, where the first term is the gradient of potential $\frac{1}{2}\mu H^2$ and the second one is orthogonal to it. The quasi-potential in this example can be computed explicitly 
as $V = \mu H^2$, and the prefactor $Z_\epsilon$ is a constant.  

Due to damping, the deterministic system associated with \eqref{eq-f8} is attracted to the energy surface $H = 0$, which admits an attractor in the shape of a figure-eight loop in the phase plane. 
The origin $(0,0)$ is a saddle point whose  stable and unstable manifolds coincide to 
form two homoclinic orbits; see the top-left panel of Figure \ref{f8-fig1} for an illustration.The stochastic system \eqref{eq-f8} exhibits an intriguing limiting behavior: as $\epsilon\to0$, trajectories  spend increasingly long periods near the saddle $(0,0)$ before following the homoclinic orbit. As a result, although the probability density $u_\epsilon$ remains non-zero
along the homoclinic orbit for  sufficiently small $\epsilon$,  the invariant probability measure $\pi_\epsilon$ converges weakly to the Dirac measure  $\pi_0 = \delta_0$  as $\epsilon \rightarrow 0.$ 

It is worth noting that the speed of the convergence of $\pi_\epsilon$ to its weak limit, as $\epsilon \rightarrow 0$, is notably slow. For general SDEs, the large deviation theory implies that 
$\pi_\epsilon(A) \approx \exp(-c/\epsilon^2)$ for some constant $c > 0$, provided that the set $A$ does not intersect the support of  the weak limit $\pi_0$. In the present example, consider an open set $A$ that intersects the homoclinic orbit but excludes the support of $\pi_0$. Since trajectories of \eqref{eq-f8} requires $O(1)$ time to traverse the homoclinic loop and $O(- \log \epsilon)$ time to pass through a neighborhood of the saddle point, it is expected that 
\[\pi_\epsilon(A) \approx O(-1/\log \epsilon).\] In other words, $\pi_\epsilon(A)$ in this example is significantly large than $\exp(-c/\epsilon^2)$. For detailed studies on the dynamics of SDEs near saddle points, 
readers are referred to \cite{bakhtin2011noisy, monter2011normal, bakhtin2010small}. In particular, as noted in \cite{young2002srb}, the figure-eight attractor does not admit an SRB measure. 

Although the stochastic dynamics associated with the figure-eight attractor differs from that of classical attractors, the proposed DeepWKB algorithm remains applicable for computing the WKB expansion of its invariant distribution. 
The value of $\sqrt\epsilon$ is taken as $0.08, 0.09,$ $0.1, 0.11, 0.12, 0.14, 0.16, 0.18, 0.2,$ and $0.24$, with each simulation run for $2\times 10^{10}$ steps using a  step size $\mathrm{d}t = 0.002$. For each $\epsilon$, the  invariant probability density  is evaluated at $5000$ collocation points. After excluding  points with insufficient samples, least square fitting is  performed at $4287$ collocation points. An additional $2404$ collocation points are generated near the figure-eight attractor, along with $713$  points far from the attractor, for which artificial output values are assigned. A scatter plot of the resulting training set is shown in the top-middle panel of Figure \ref{f8-fig1}. 

To improve training performance, the training set for $V$ is expanded to regions farther from the attractor along the characteristic curves. Specifically, an additional $273$ collocation points are selected near the level set $V=0.06$, where the partial derivatives of $V_{\bm \theta}$ are computed. These values serve as initial conditions 
for computing the characteristic curves of the Hamilton-Jacobi equation using the symplectic Euler method. 
Along each characteristic curve, $15$ additional  collocation points are generated, forming an augmented training set; see the top-right panel of Figure \ref{f8-fig1}. The resulting trained $V_{\bm \theta}$ on the expanded training set is shown in the bottom-right panel of Figure \ref{f8-fig1}. 

The trained $V_{\bs \theta}$ is  compared with its theoretical counterpart $\mu H^2$. As shown in the bottom-middle panel of Figure \ref{f8-fig1}, the approximation error is less than $0.01$ in a neighborhood of the attractor. The trained $V_{\bm \theta}$ is then used to compute the function $Z_0$. As descried in Section \ref{sec:Z}, 
initial estimates of $Z_0$ near the attractor are obtained via linear regression on  Monte Carlo data of the invariant distributions 
at $3000$ collocation points.  The approximated invariant density function   \[u_\epsilon=Q(\epsilon)^{-1}Z_0 \exp\{-V/\epsilon\}\] with $\sqrt{\epsilon} = 0.15$ is shown in the bottom-right panel of Figure \ref{f8-fig1}.

\begin{figure}[htbp]
    \centering    \includegraphics[width=\linewidth]{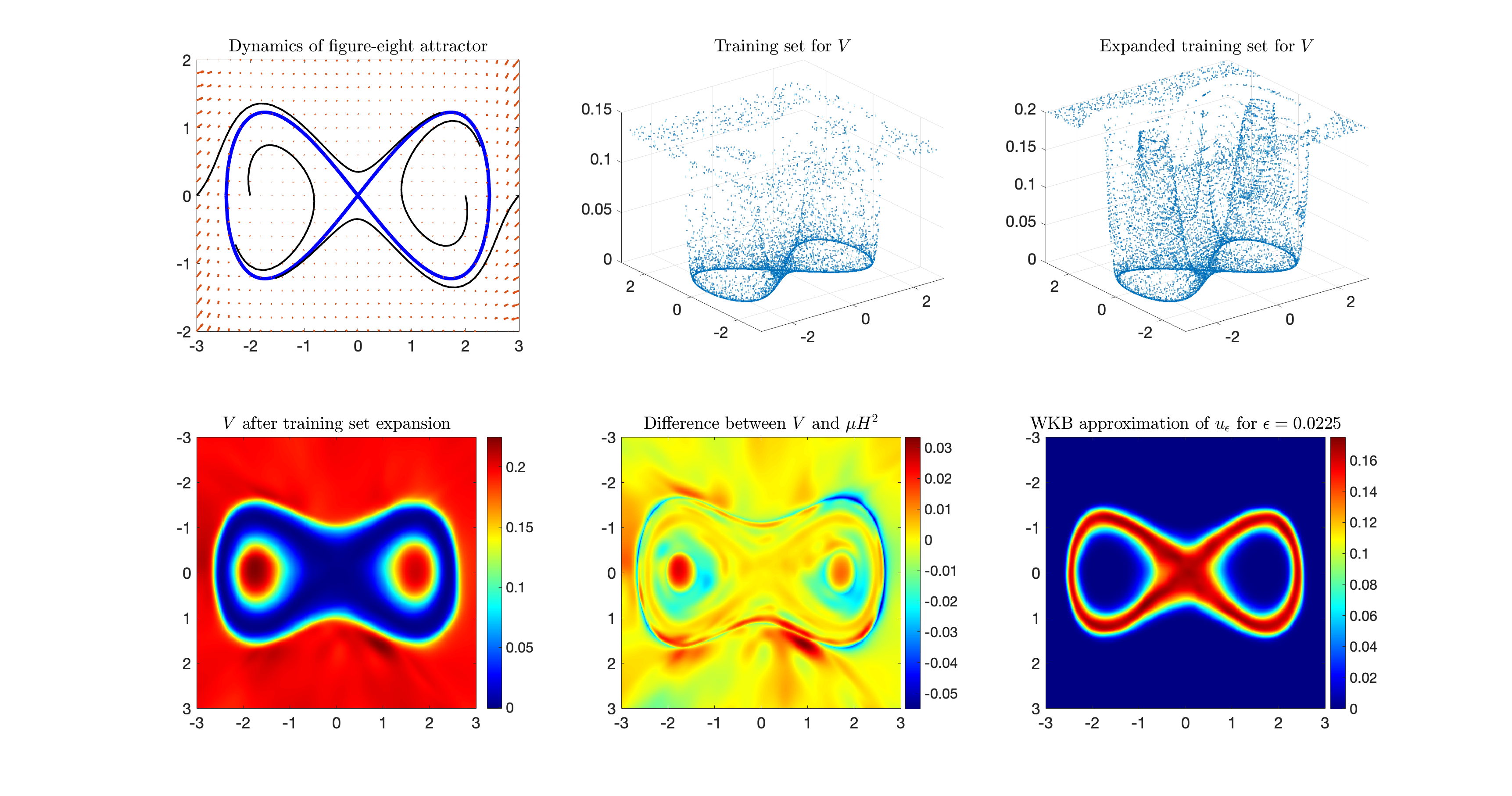}
    \caption{Top Left: Sketch of deterministic dynamics of figure-eight attractor. Top Middle: 3D scatter plot of training set for $V$. Top Right: 3D scatter plot of expanding training set for $V$. Bottom Left: Training result of $V$ after training set expansion. Bottom Middle: Error plot comparing the trained $V$ and its theoretical counterpart $\mu H^2$. Bottom Right: WKB approximation $u_\epsilon =Q(\epsilon)^{-1} Z_0\exp\{-V/\epsilon\}$ for $\epsilon = 0.0225$.}
    \label{f8-fig1}
\end{figure}

\subsection{Stochastic coupled Van der Pol oscillators}
The example in this subsection considers a system of two coupled Van der Pol oscillators
\begin{eqnarray}\label{eq-VDP4}
\left\{  
    \begin{aligned} 
    &\mathrm{d}X_1 = \big(\mu(X_1 - \frac{1}{3}X_1^3 - Y_1) + \delta (X_1 - X_2)\big) \mathrm{d}t + \sqrt{\epsilon} \mathrm{d}B^{(1)}_t\\
    &\mathrm{d}Y_1 = \frac{1}{\mu}X_1 \mathrm{d}t + \sqrt{\epsilon} \mathrm{d}B^{(2)}_t \\
    &\mathrm{d}X_2 = \big(\mu(X_2 - \frac{1}{3}X_2^3 - Y_2) + \delta (X_2 - X_1) \big)\mathrm{d}t + \sqrt{\epsilon} \mathrm{d}B^{(3)}_t\\
    &\mathrm{d}Y_2 = \frac{1}{\mu}X_2 \mathrm{d}t + \sqrt{\epsilon} \mathrm{d}B^{(4)}_t
\end{aligned}\right.
\end{eqnarray}
where $\mu$ denotes the scaling parameter of nonlinear damping, and $\delta$ denotes the coupling strength between the two oscillators. 

 This example aims to demonstrate that the DeepWKB method is applicable to higher-dimensional problems whose attractors are  invariant manifolds rather than simple limit cycles. Specifically, the fully decoupled case $\delta = 0$ is considered, for which the deterministic part of equation \eqref{eq-VDP4} admits a stable invariant manifold $S\times S$, where $S$ denotes the limit cycle of the two-dimensional Van der Pol oscillator. From a numerical perspective, introducing  coupling ($\delta \neq 0$) simplifies the problem by inducing  phase-locking between the oscillators, thereby reducing the invariant manifold to a  single limit cycle. Moreover,  the fully decoupled case  provides results that are easier to visualize.  

As in the two-dimensional Van der Pol oscillator, three classes of collocation points are selected: points on the attractor, points near the attractor for the WKB approximation, and points distributed over the entire numerical domain for minimizing  the operator residual. Monte Carlo simulations are  conducted for $\sqrt\epsilon = 0.08, 0.09, 0.1, 0.11, 0.12, 0.14, 0.16, 0.18, 0.2, $ and $0.24$, and the invariant  density function is estimated at $50000$ collocation points using the mesh-free sampler developed in \cite{zhai2022deep}. After discarding collocation points with insufficient samples, least square fitting is applied at the remaining $38056$ points to obtain  estimation of $V$, $Z_0$ and $Z_1$.  

An additional $10000$ collocation points are randomly sampled on the invariant manifold, along with $10000$ collocation points farther  from the attractor, where  artificial value are prescribed. A three-dimensional scatter plot of the collocation points $(X_1, X_2, V)$ is  shown in the left panel of Figure \ref{vdp4-fig1}. To 
visualize the four-dimensional function, two-dimensional slices are taken at 
\begin{eqnarray}\label{slice}
\quad\big\{(-0.064, 1.58, x, y)\,|\, x, y \in \mathbb{R}\big\}\ \ \text{and}\ \ \big\{(x, y, 1.800, 0.166)\,|\, x, y \in \mathbb{R}\big\}.    
\end{eqnarray}
The fixed coordinates in each slice are randomly selected along the limit cycle of the two-dimensional  Van der Pol oscillator. The trained quasi-potential functions on these two  slices are displayed in the  middle and right panels of Figure \ref{vdp4-fig1}. 

\begin{figure}[htbp]
    \centering
    \includegraphics[width=\linewidth]{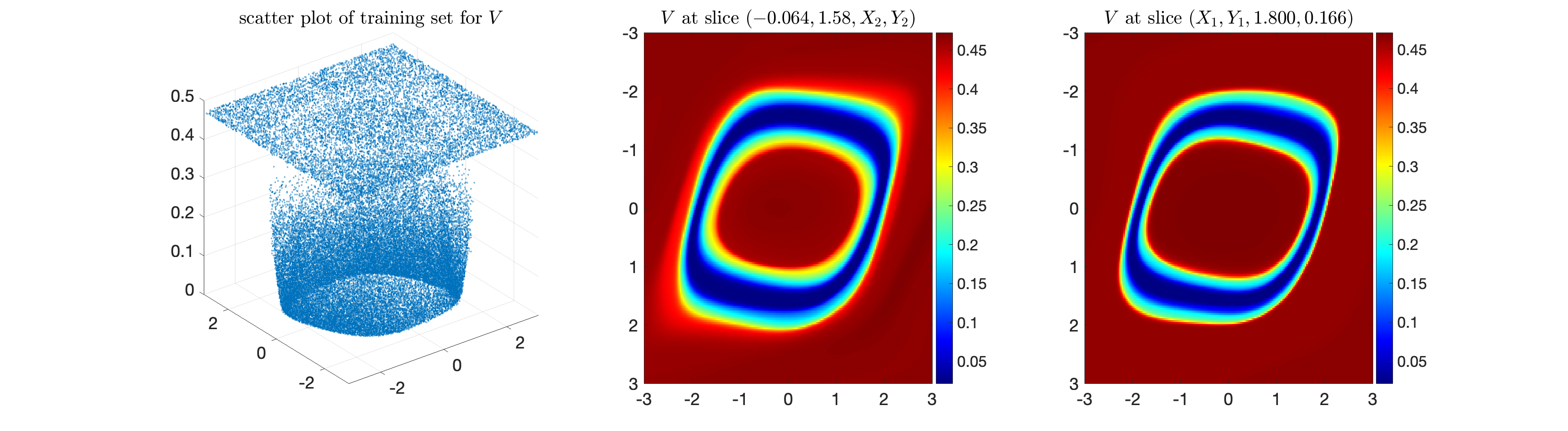}
    \caption{Left: 3D scatter plot of $(X_1, X_2, V)$-component of the training set. Middle and Right: Trained quasi-potential function $V$ on  two different two-dimensional slices.}
    \label{vdp4-fig1}
\end{figure}

Next, $600$ points are randomly selected on the level set of $V = 0.06$, and the partial derivatives of $V_{\bm \theta}$ are evaluated at these points. These positions and derivatives are then used to compute the characteristic curves of the associated Hamilton-Jacobi equation, thereby expanding the input set of the training data to the level set $V = 0.2$. Using the previously obtained parameter ${\bm \theta}$ as an initial guess,  the training of $V_{\bs\theta}$ is refined on the augmented training set. A three-dimensional scatter plot of the extended training set and the trained quasi-potential function, along with the trained quasi-potential function restricted to the 
 two-dimensional slices \eqref{slice}, is shown in  the  top three panels of Figure \ref{vdp4-fig2}. 

Finally, the trained quasi-potential  $V_{\bm \theta}$ is used to train the prefactor function $Z_0$. The bottom-left panel of Figure \ref{vdp4-fig2}  displays the values of $Z_0$ restricted to the two-dimensional slice where the last two coordinates are fixed at $(1.800, 0.166)$. In this example, since the training set for $Z_0$ is limited to the neighborhood of the attractor, the resulting approximation is valid only near the attractor. Once $Z_0$ is obtained, the WKB ansatz is employed to approximate the invariant probability density $u_\epsilon$ for $\epsilon = 0.0225$. The WKB estimation, given by  
$Q(\epsilon)^{-1}Z_0 \exp\{-V/\epsilon\}$, evaluated on the two aforementioned two-dimensional slices, is shown in the bottom middle and right panels of Figure~\ref{vdp4-fig2}.
It is observe that, despite  the lack of training data for $Z_0$ beyond the vicinity of the attractor, the WKB approximation captures the invariant probability density $u_\epsilon$  quite well in the small-noise regime.

\begin{figure}[htbp]
    \centering
    \includegraphics[width=\linewidth]{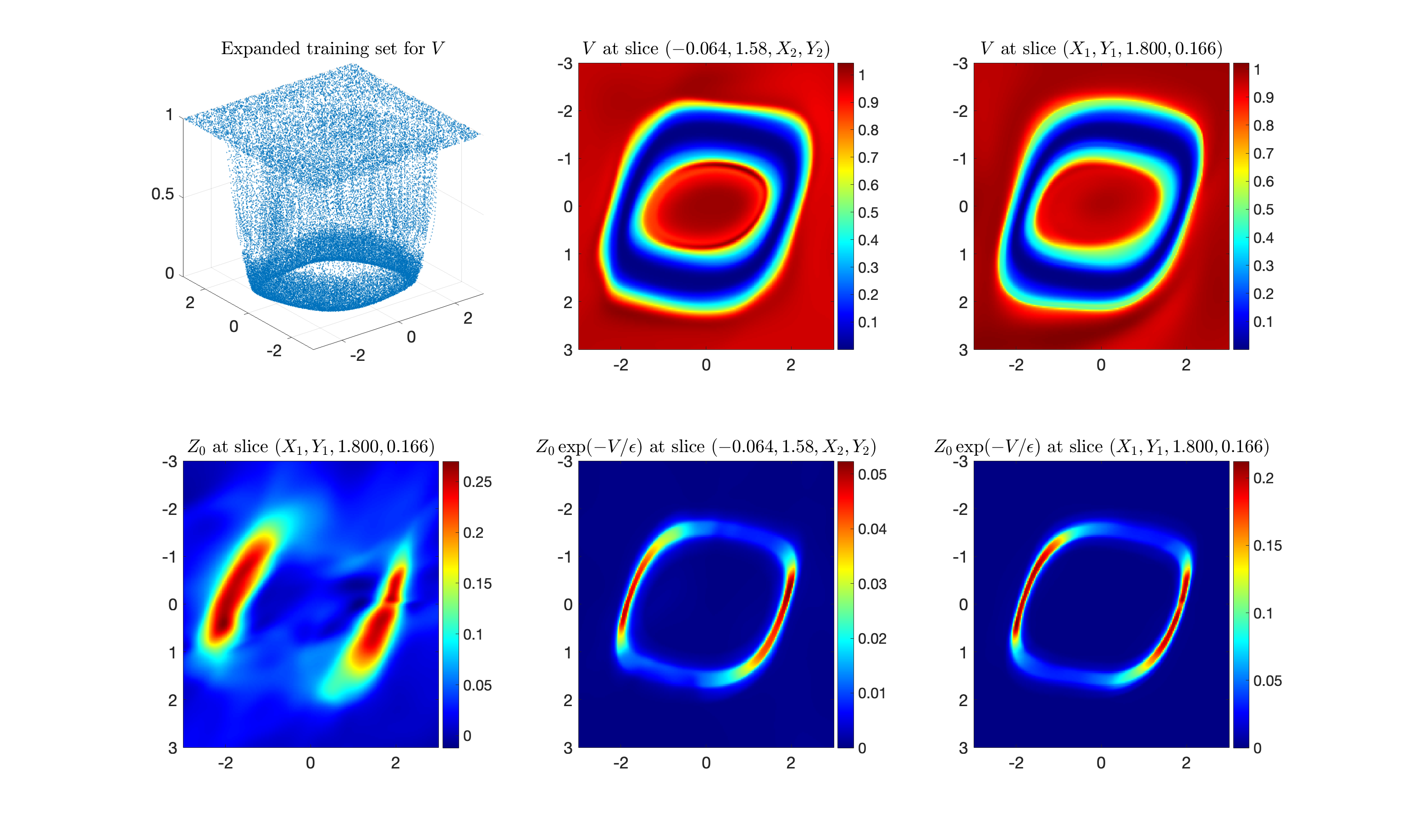}
    \caption{Top Left: 3D scatter plot of $(X_1, X_2, V)$ after training set expansion. Top Middle and Left: Trained $V$ at two 2D slices after training set expansion. Bottom Left: $Z_0$ on the 2D slice $\{(x, y, 1.800, 0.166)\,|\, x, y \in \mathbb{R}\}$. Bottom Middle and Right: WKB approximations $u_\epsilon \approx Q(\epsilon)^{-1}Z_0\exp\{-V/\epsilon\}$ restricted on two  2D slices for $\epsilon = 0.0225$.}
    \label{vdp4-fig2}
\end{figure}

\subsection{Stochastic R\"ossler oscillator}
This subsection investigates  the WKB approximation of the invariant distribution for the stochastic R\"ossler oscillator
\begin{eqnarray}\label{eq-rossler}
\left\{
\begin{aligned}    
    \mathrm{d}X_t &= (- Y_t - Z_t) \mathrm{d}t + \sqrt{\epsilon} \mathrm{d}B^{(1)}_t\\
    \mathrm{d}Y_t &= ( X_t + a Y_t) \mathrm{d}t + \sqrt{\epsilon} \mathrm{d}B^{(1)}_t\\
    \mathrm{d}X_t &= (b +  Z_t(X_t - c)) \mathrm{d}t + \sqrt{\epsilon} \mathrm{d}B^{(1)}_t
\end{aligned}
\right.
\end{eqnarray}
with the classical parameters $a = 0.2, b = 0.2$, and $c = 5.7$m under which the deterministic R\"ossler system exhibits chaotic dynamics; see the top-left panel of Figure \ref{rossler-fig}. 
Until now, there has been  no rigorous theoretical justification for the applicability of the WKB approximation to the invariant distributions of the stochastic R\"ossler oscillator. 
However, as shown in the next subsection, statistical verification over the chosen range of $\epsilon$ confirms that the invariant distribution admits a WKB expansion. Thus, the quasi-potential function $V$ and the prefactor function $Z$ can be computed as in the previous examples.   

The noise parameters are chosen as $\sqrt{\epsilon} = 0.1, 0.11,$ $0.12, 0.14, 0.16$, $0.18,$ $0.2, 0.22,$ $ 0.24,$ and $0.28$. Each simulation is performed for $4 \times 10^{10}$ steps with a step size $\mathrm{d}t = 0.001$. Unlike the previous  examples, 
the stochastic R\"ossler oscillator does not possess a stationary  distribution, as its trajectories retain a small probability of escaping to infinity. Instead, it 
admits a quasi-stationary distribution (QSD), defined by conditioning on non-escape. 
Due to the low escaping rate,  the WKB expansion of the QSD in the same form can still be investigated. 

The mesh-free sampling method proposed in  \cite{zhai2022deep} is employed to estimate the QSD at $60000$ collocation points. All collocation points lie within the box $[-15, 15]\times [-1.5, 1.5] \times[-1.5, 1.5]$, which does not include the top arch of the R\"ossler oscillator. After discarding collocation points with insufficient  samples, least square fitting is applied to estimate the functions $V$, $Z_0$, and $Z_1$ at  the remaining $50007$ collocation points. Additionally, $20000$ collocation points are randomly selected on the deterministic attractor, along with $10000$ additional points farther from the attractor, where artificial values are prescribed.
Two three-dimensional scatter plots 
are shown  in Figure \ref{rossler-fig}: the top-middle panel depicts the distribution of the collocation points, while the top-right panel illustrates the $(x, y, V)$-component of the training set. 
The trained quasi-potential function evaluated at two planar slices with $z = -0.05$ and $z = 0.05$ is shown in the bottom-left and bottom-middle panels of Figure \ref{rossler-fig}.     

Finally, the trained quasi-potential function is used to train the prefactor function $Z_0$. 
For $\epsilon = 0.075$, the variable $z$ in the expression  
$Z_0 \exp\{-V/\epsilon\}$ is integrated over the interval $[-0.1,0.1]$, covering the ``disk''-shape region of the R\"ossler oscillator. Define 
$$
    \Phi(x,y) = \int_{-0.1}^{0.1} Z_0(x,y,z) \exp(-V(x,y,z)/\epsilon) \mathrm{d}z.
$$
The heat map of $\Phi$ is shown in the bottom-right panel of Figure \ref{rossler-fig}. 

\begin{figure}[htbp]
    \centering    \includegraphics[width=\linewidth]{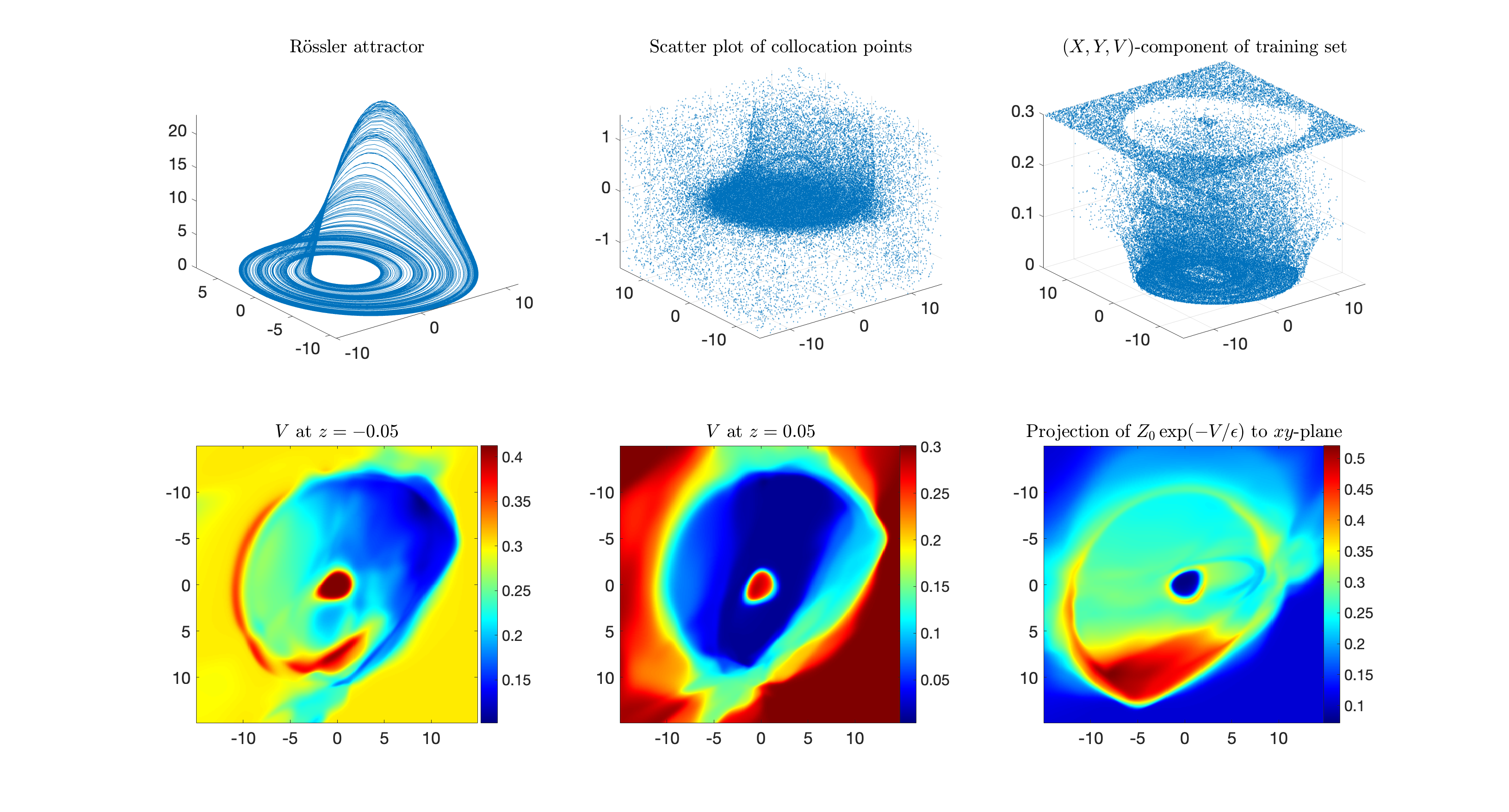}
    \caption{Top Left: the R\"ossler attractor. Top Middle: Scatter plot of collocation points. Top Right: Scatter plot of $(X, Y)$-component of the training set and the corresponding $V$-values. Bottom Left and Middle: Training  results for the quasi-potential function $V$ at $z = -0.05$ and $z = 0.05$. Bottom Right: WKB approximation $u_\epsilon \approx Q(\epsilon)^{-1}Z_0 \exp\{-V/\epsilon\}$ for $\epsilon = 0.075$, projected onto the $xy$-plane.}
    \label{rossler-fig}
\end{figure}

The training results for the stochastic R\"ossler oscillator are less satisfactory compared to the previous three examples. The quasi-potential function fails to capture the 
fractal structure present in the ``disk'' region of the strange attractor, and expanding the training set does not yield significantly improvement in the accuracy of $V$. Consequently, the prefactor $Z_0$ is less accurate than  in the other three examples. 
A possible explanation is that the attractor of the R\"ossler oscillator covers a large domain, and neural networks generally approximate functions less effectively over such extensive regions. This limitation arises from the near-zero derivatives of activation functions at large input values, which obstruct the training process. 
While rescaling the domain to $[-1.5, 1.5]^3$ could  mitigate this issue, it artificially amplifies the derivatives of both $V$ and $Z_0$, resulting in poorer neural network approximations.
Moreover, since the R\"ossler attractor is a strange attractor  rather than a smooth manifold, selecting  suitable initial points for the characteristic curves to  achieve an even coverage of the domain is a substantial challenge. 

\subsection{Statistical verification of  WKB approximations}
\label{sec:WKB_verify}
In this subsection, the statistical method proposed in Section \ref{sec:WKB_statistical_validation} is employed to assess the validity of  the WKB approximations for the 
numerical examples.

For the first three examples,  statistical verification of WKB is presented only for the stochastic Van der Pol oscillator. In this case, the deterministic system admits a stable limit cycle, and the 
validity of the WKB approximation is theoretically guaranteed in the small noise limit. This helps to confirm the effectiveness of the 
statistical method.  
Verifications for the other  two examples proceed analogously. 

\vspace{1mm}

{\bf Stochastic Van der Pol oscillator.} 
Monte Carlo simulation is employed to statistically verify the WKB expansion of the stochastic Van der Pol oscillator \eqref{eq-VDP}. To mitigate the effect of bin size, samples are collected on a  $2048\times 2048$ bin grid over the  domain $[-3, 3]\times[-3, 3]$. The noise magnitudes are set to $\sqrt\epsilon = 0.04, 0.05, 0.06, 0.07, 0.08, 0.09, 0.1, 0.12, 0.14$, and $0.16$ respectively, with the smallest $\epsilon$ at least $10$ times larger than the bin size. 

For each value of $\epsilon$, 10 independent trajectories are simulated in parallel up to time $t = 4 \times 10^7$ via  Monte Carlo method using a time step size $\mathrm{d}t = 0.002$. Samples are collected every
$\mathrm{d}T = 0.02$, corresponding to one sample every $10$ steps under the Euler-Maruyama scheme. The resulting data are analyzed using Algorithm \ref{alg:WKB}. As shown in  Figure \ref{fig1}, the RSS $\| \mathbf{r}\|^2$ displays no apparent spatial concentration (top-left panel), and its normalized histogram closely follows the $\chi^2$-distribution with $7$ degrees of freedom (top-right panel). 

For comparison, the Monte Carlo simulation is repeated using a smaller
sampling interval $\mathrm{d}T = 0.002$. The corresponding results are illustrated in the two bottom panels of Figure \ref{fig1}. It is observed that the RSS values increase due to a reduction in the number of  effective samples relative to the total number  collected. Nevertheless, the empirical distribution of the RSS remains qualitatively similar to that shown in the top panels of Figure \ref{fig1}. 
This provides convincing evidence that the stochastic Van der Pol oscillator has a WKB expansion for the tested range of $\epsilon$. This  confirms the consistency of the proposed statistical verification method with the theoretical predictions.

\begin{figure}[htbp]
    \centering
    \includegraphics[width=\linewidth]{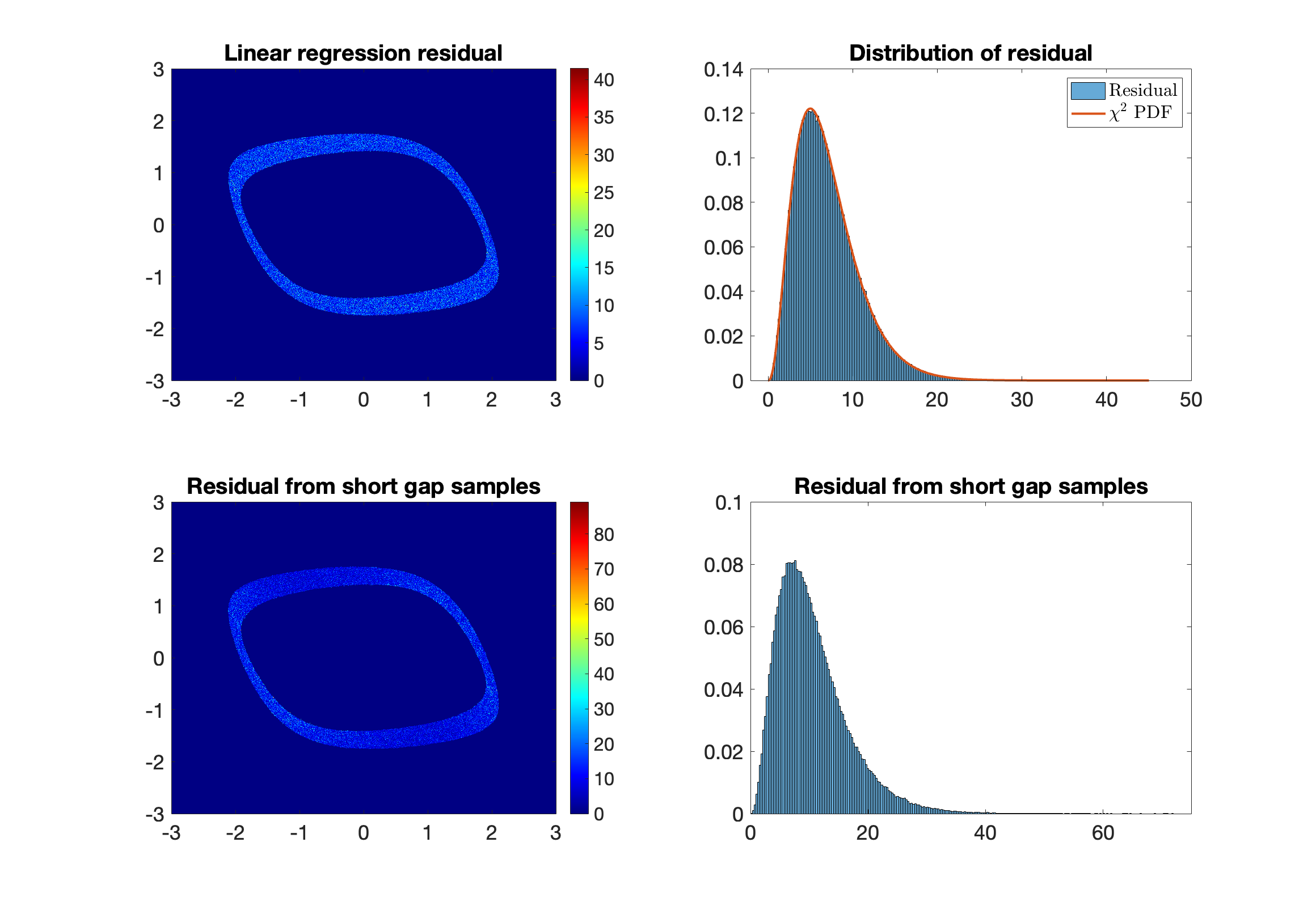}
    \caption{Top Left: RSS at collocation points. Top Right: normalized histogram of RSS vs $\chi^2$ distribution with $7$ degree of freedom. Bottom: RSS at collocation points and histogram with small sampling interval between Monte Carlo samples.}
    \label{fig1}
\end{figure}

\vspace{1mm}

{\bf Stochastic R\"ossler oscillator.} The WKB expansion of the stochastic R\"ossler oscillator  \eqref{eq-rossler} is also statistically verified. Due to the difficulty of visualizing a three-dimensional probability density function, the  density is estimated on the slice $Z = 0.01$ using a $3000\times 3000$ bin mesh covering the square $-15 \leq X, Y \leq 15$. The noise magnitudes are chosen as $\sqrt\epsilon = 0.1, 0.11, 0.12, 0.14,$$0.16, 0.18,$ $0.2, 0.22,$ $0.24$, and $0.28$, with the smallest $\epsilon$ at least $10$ times of the bin size. 
Samples are collected at intervals $\mathrm{d}T = 0.02$, $10$ times the Euler-Maruyama time step size. 

For each $\epsilon$,  $48$ independent trajectories are simulated in parallel up to time $t = 10^7$. The RSS $\| \mathbf{r}\|^2$ is computed  at each collocation point with enough samples and displayed in the top-left panel of Figure \ref{fig2}. The RSS exhibits no apparent spatial structure,  aside from an arch-shaped region with slightly higher values. 
The normalized histogram of RSS matches well with the $\chi^2$ distribution with $7$ degrees of freedom; see top-right panel of Figure \ref{fig2}. The least square fit of $D \hat{\mathbf{u}}$ versus  $\epsilon$ at  collocation points with the largest residual is shown in the bottom-left panel of Figure \ref{fig2}, demonstrating reasonable accuracy despite some fluctuations. For comparison, the bottom-right panel displays the RSS distribution obtained using a reduced sampling interval $\mathrm{d}T = 0.002$. In this case, the residual increase substantially in certain regions, underscoring the importance of using sufficiently large sampling intervals when verifying the WKB approximation. 

Based on the results presented above,  the WKB approximation is expected to hold for the stochastic R\"ossler oscillator, at least within the tested range of $\epsilon$. It should be noted, however, that the simulation results do  not completely exclude the possibility that the fractal structure of the strange attractor may cause a breakdown of the WKB approximation in the limit of vanishing noise. 

\begin{figure}[htbp]
    \centering
    \includegraphics[width=\linewidth]{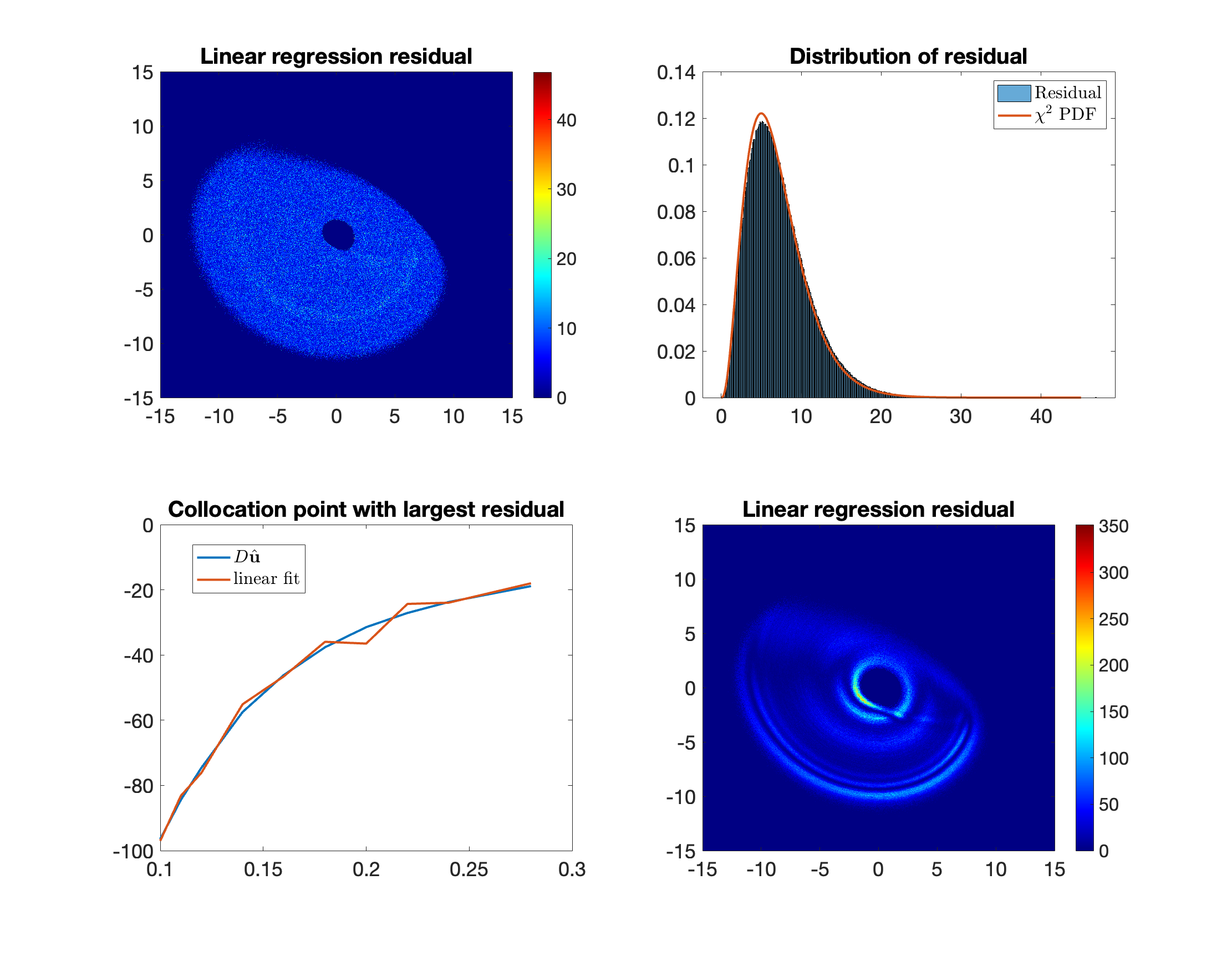}
    \caption{Top Left: RSS at collocation points. Top Right: normalized histogram of RSS versus $\chi^2$ distribution with $7$ degrees of freedom. Bottom Left: Example of least square fitting 
    of $D\hat{\mathbf u}$ vs $\epsilon$  at a collocation point with largest residual. Bottom Right: RSS for simulations with small sampling interval.}
    \label{fig2}
\end{figure}

\section{Conclusion and further discussions}
This paper presents a novel machine learning-based approach, termed the DeepWKB method, for numerically computing the invariant probability density $u_\epsilon$ 
of the  stochastic  system \eqref{SDE} via its WKB approximation \[u_\epsilon(\bs x)=Q(\epsilon)^{-1}Z_\epsilon(\bs x)\exp\{-V(\bs x)/\epsilon\},\] where $\epsilon$ denotes the noise strength. The function $V$, known as the Freidlin-Wentzell quasi-potential, plays a central role in the study of rare events, metastability, transition phenomena, and first-exit problems. The key innovation of the proposed method is the simultaneous estimation of both functions $V$ and  $Z_\epsilon$
near the attractor, using these values as training data for neural networks, without imposing any assumptions on the attractor’s geometric structures. This achieved by applying  linear regression to Monte Carlo samples of  invariant densities 
at multiple values of $\epsilon$. 
The training set is further augmented by integrating $V$ and $Z$ along their characteristic curves, which leads to accurate approximation of them by neural networks. In contrast, most existing methods for computing the quasi-potential $V$  assume  the attractor to be either  a stable equilibrium or a stable limit cycle. 

The DeepWKB method is tested on several numerical examples:  the Van der Pol oscillator with a stable limit cycle, a figure-eight attractor consisting of two homoclinic orbits from a saddle point,
a coupled Van der Pol oscillator with a stable invariant manifold,  and the R\"ossler oscillator with a strange attractor. Since the existence of WKB approximation is generally not guaranteed, a statistical tool is developed to assess its validity. In particular, it is shown that despite  the presence of a strange attractor, the invariant probability density of the stochastic R\"ossler oscillator statistically satisfies the WKB approximation, at least for the tested range of noise magnitudes. 

The DeepWKB method can be further enhanced in two directions. The first  improvement concerns 
the sampling technique. Currently, 
 the probability density function $u_\epsilon$ is sampled using a 
mesh-free Monte Carlo sampler developed in \cite{zhai2022deep}. 
A major limitation of this approach is the difficulty in collecting sufficient samples in high-dimensional settings. With suitable modifications, the forward-reverse sampler proposed in \cite{milstein2004transition} could be adapted to improve sampling quality for SDEs, particularly in approximating  invariant probability distributions. 
The second direction involves expanding the training set to larger domains. It has been observed that characteristics of the  Hamilton-Jacobi equation associated with $V$ may intersect when sufficiently far from the attractor. Consequently, to compute the WKB approximation over arbitrarily large domains, it is necessary to handle shocks and employ algorithms such as weak PINNs for approximating  entropy solutions \cite{de2024wpinns}.

\bibliographystyle{abbrv}
\bibliography{references}

\end{document}